\documentclass[11pt]{amsart}
\usepackage{amssymb,graphicx}
\usepackage{amsfonts}
\usepackage{amsmath}
\usepackage{amsthm}
\usepackage{fancyhdr}
\usepackage{indentfirst}
\usepackage{hyperref}
\usepackage{comment}
\usepackage{color}
\usepackage{subcaption}
\usepackage{epsfig}
\usepackage{pst-grad} % For gradients
\usepackage{pst-plot} % For axes
\usepackage[space]{grffile} % For spaces in paths
\usepackage{etoolbox} % For spaces in paths
\usepackage[normalem]{ulem} %For strikethrough text

\usepackage{soul}

\newcommand{\al}{\alpha}

\newcommand{\V}{\mathcal{V}}
\newcommand{\R}{\mathbb{R}}

\newcommand{\la}{\lambda}

\newcommand{\La}{\Lambda}
\newcommand{\si}{\sigma}
\newcommand{\Si}{\Sigma}

\newcommand{\pr}{\prime}
\newcommand{\Om}{\Omega}

\newcommand{\ga}{\gamma}
\newcommand{\ka}{\kappa}

\newcommand{\ti}{\tilde}
\newcommand{\wti}{\widetilde}

\newcommand{\mL}{\mathcal{L}}

\newcommand{\tB}{\widetilde{B}}
\newcommand{\tA}{\widetilde{A}}

\newcommand{\dist}{\operatorname{dist}}
\newcommand{\Div}{\operatorname{div}}
\newcommand{\tr}{\operatorname{tr}}

\newcommand{\Ric}{\operatorname{Ric}}
\newcommand{\Area}{\operatorname{Area}}

\newcommand{\Hom}{\operatorname{Hom}}

\newcommand{\mc}{\mathcal}

\begin{document}

\newtheorem{theorem}{Theorem}[section]
\newtheorem{proposition}[theorem]{Proposition}
\newtheorem{corollary}[theorem]{Corollary}

\newtheorem{claim}{Claim}

\theoremstyle{remark}
\newtheorem{remark}[theorem]{Remark}

\theoremstyle{definition}
\newtheorem{definition}[theorem]{Definition}

\theoremstyle{plain}
\newtheorem{lemma}[theorem]{Lemma}

\numberwithin{equation}{section}

\newtheorem{conjecture}[theorem]{Conjecture}

%Page head and Page foot
%\pagestyle{headings}
%\lhead{} \chead{min-max minimal surfaces} \rhead{}
%\lfoot{} \cfoot{\thepage} \rfoot{}
%\renewcommand{\headrulewidth}{0.4pt}

%Tile and Author
%\title{\textbf{Free boundary min-max theory 1: Curvature estimates}}
\title[Curvature estimates near free boundary]{Curvature estimates for stable free boundary minimal hypersurfaces}
\author[Qiang Guang]{Qiang Guang}
\address{Department of Mathematics, University of California Santa Barbara, Santa Barbara, CA 93106, USA}
\email{guang@math.ucsb.edu}

\author[Martin Li]{Martin Man-chun Li}
\address{Department of Mathematics, The Chinese University of Hong Kong, Shatin, N.T., Hong Kong}
\email{martinli@math.cuhk.edu.hk}

\author[Xin Zhou]{Xin Zhou}
%\address{Massachusetts Institute of Technology, Department of Mathematics, 77 Massachusetts Avenue, Cambridge, MA 02139, USA}
\address{Department of Mathematics, South Hall 6501, University of California Santa Barbara, Santa Barbara, CA 93106, USA}
\email{zhou@math.ucsb.edu}
\maketitle

\pdfbookmark[0]{}{beg}

%Begin abstract
%\renewcommand{\abstractname}{}    % clear the title
%\renewcommand{\absnamepos}{empty} % originally center
\begin{abstract}
In this paper, we prove uniform curvature estimates for immersed stable free boundary minimal hypersurfaces satisfying a uniform area bound, which generalizes the celebrated Schoen-Simon-Yau interior curvature estimates \cite{SSY75} up to the free boundary. Our curvature estimates imply a smooth compactness theorem which is an essential ingredient in the min-max theory of free boundary minimal hypersurfaces developed by the last two authors \cite{LiZ16}. We also prove a monotonicity formula for free boundary minimal submanifolds in Riemannian manifolds for any dimension and codimension. For $3$-manifolds with boundary, we prove a stronger curvature estimate for properly embedded stable free boundary minimal surfaces without a-prioi area bound. This generalizes Schoen's interior curvature estimates \cite{Schoen83} to the free boundary setting. 
Our proof uses the theory of minimal laminations developed by Colding and Minicozzi in \cite{CM13}.
\end{abstract}

%%%%%%%%%%%%%%%%%%%%%%%%%%%%%%%%%%
% Section 1   	 		  Introduction	                      		     %
%%%%%%%%%%%%%%%%%%%%%%%%%%%%%%%%%%
\section{Introduction}
\label{S:intro}

Let $(M,g)$ be an $m$-dimensional Riemannian manifold, and $N$ be an embedded $n$-dimensional submanifold called the \emph{constraint submanifold}. If we consider the $k$-dimensional area functional on the space of immersed $k$-submanifolds $\Si \subset M$ with boundary $\partial \Sigma$ lying on the constraint submanifold $N$, the critical points are called \emph{free boundary minimal submanifolds}. These are minimal submanifolds $\Si \subset M$ meeting $N$ orthogonally along $\partial \Si$ (c.f. Definition \ref{D:freeboundary}). Such a critical point is said to be \emph{stable} (c.f. Definition \ref{D:stable}) if it minimizes area up to second order. %i.e. the second variation is non-negative. 
The purpose of this paper is three-fold. First, we prove uniform curvature estimates (Theorem \ref{T:main}) for \emph{immersed} stable free boundary minimal hypersurfaces satisfying a uniform area bound. Second, we prove a monotonicity formula (Theorem \ref{T:monotonicity}) near the boundary for free boundary minimal submanifolds in any dimension and codimension. %This same formula also holds in the varifold setting (Theorem \ref{T:monotonicity-varifold}). 
Finally, we use Colding-Minicozzi's theory of minimal laminations (adapted to the free boundary setting) to establish a stronger curvature estimate (Theorem \ref{T:curvature estimates 2-D}) for \emph{properly embedded} stable free boundary minimal surfaces in compact Riemannian $3$-manifolds with boundary, \emph{without} assuming a uniform area bound on the minimal surfaces.

Curvature estimates for immersed stable minimal hypersurfaces in Riemannian manifolds were first proved in the celebrated work of Schoen, Simon and Yau in \cite{SSY75}. Such curvature estimates have profound applications in the theory of minimal hypersurfaces. For example, Pitts \cite{P81} made use of Schoen-Simon-Yau's estimates in an essential way to establish the regularity of minimal hypersurfaces $\Si$ constructed by min-max methods, for $2 \leq \dim \Si \leq 5$ due to the dimension restriction in \cite{SSY75}. Shortly after, Schoen and Simon \cite{SS81} generalized these curvature estimates to any dimension (but still for codimension one, i.e. \emph{hypersurfaces}) for \emph{embedded} stable minimal hypersurfaces, which enabled them to complete Pitts' program for $\dim \Si >5$.

In this paper, we establish uniform curvature estimates in the free boundary setting. The theorem below follows from our curvature estimates near the free boundary (Theorem \ref{T:main-curvature-estimates}) and the interior curvature estimates \cite{SSY75}.

%Such a formula in the free boundary setting was first given in \cite{GJ86} where they used it to prove a Allard-type regularity theorem for stationary varifolds with free boundary. 

%Courant \cite{Courant} initiated the study of free boundary minimal disks with boundary lying on an embedded closed surface $N \subset \R^3$.

%The purpose of this paper is to prove several uniform curvature estimates for minimal hypersurfaces of $M$ with free boundary lying on $N$. For our application in \cite{LiZ16}, we are particularly interested in the case that $N$ bounds a compact domain $\Omega \subset M$, i.e. $\partial \Omega =N$. \footnote{As in \cite{LiZ16}, any compact Riemannian manifold $\Omega$ with boundary $\partial \Omega \neq \emptyset$ can be extended to a closed Riemannian manifold $M$.} 

%Combining with the interior curvature estimates for two-sided stable minimal hypersurfaces by Schoen-Simon-Yau \cite{SSY75}, we have the following corollary.

\begin{theorem}
\label{T:main}
Assume $2 \leq n \leq 6$. 
Let $M^{n+1}$ be a Riemannian manifold and $N^n \subset M$ be an embedded hypersurface. Suppose $U \subset \subset M$ is an open subset. 
If $(\Si,\partial \Si) \subset (U,N \cap U)$ is an immersed (embedded when $n=6$) stable (two-sided) free boundary minimal hypersurface with $\Area(\Si) \leq C_0$, then
\[ |A^\Sigma|^2 (x) \leq \frac{C_1}{\dist^2_M(x,\partial U)} \quad \text{ for all $x \in \Sigma$},\]
where $C_1>0$ is a constant depending only on $C_0$, $U$ and $N \cap U$.
\end{theorem}

%\begin{remark}
%{\color{blue} Discuss the application for manifold with boundary $(M, \partial M)$.}
%\end{remark}

An important consequence of Theorem \ref{T:main} is a smooth compactness theorem for stable free boundary minimal hypersurfaces which are \emph{almost properly embedded} (c.f. \cite[Theorem 2.15]{LiZ16}). As in \cite{P81}, this is a key ingredient in the regularity part of the min-max theory for free boundary minimal hypersurfaces in compact Riemannian manifolds with boundary, which is developed in \cite{LiZ16} by the last two authors. We remark that any compact Riemannian manifold $\Omega$ with boundary $\partial \Omega = N$ can be extended to a closed Riemannian manifold $M$ with $\Omega$ as a compact domain. Hence, our curvature estimates above can be applied in this situation as well.

%\begin{corollary}
%\label{C:main}
%\color{blue}{Assume $2 \leq n \leq 5$. Let $M^{n+1}$ be a compact Riemannian manifold with smooth nonempty boundary, and $U \subset M$ be a relatively open subset. If $\Sigma \subset U$ is a properly embedded (i.e. $\Sigma \cap \partial M=\partial \Sigma$) stable (two-sided) free boundary minimal hypersurface with $\Area(\Si) \leq C_0$, then
%\[ |A^\Sigma|^2 (x) \leq \frac{C_1}{\dist^2_M(x,\partial U \setminus \partial M)} \quad \text{ for all $x \in \Sigma$}\]
%where $C_1>0$ is a constant depending only on $C_0$, $U$ and $M$.}
%\end{corollary}

Our proof of the curvature estimates uses a contradiction argument. If the curvature estimates do not hold, we can apply a blow-up argument to a sequence of counterexamples together with a reflection principle to obtain a non-flat complete stable immersed minimal hypersurface $\Si_\infty$ in $\R^{n+1}$ \emph{without boundary}. We then apply the Bernstein Theorem in \cite[Theorem 2]{SSY75} (which only holds for $2 \leq n \leq 5$) or \cite[Theorem 3]{SS81} (when $n=6$ for embedded hypersurface) to conclude that $\Si_\infty$ is flat, hence resulting in a contradiction. Using Ros's estimates \cite[Theorem 9 and Corollary 11]{Ros06} for one-sided stable minimal surfaces, our result also holds true when $n=2$ if one removes the two-sided condition. When $n \geq 7$, the stable free boundary minimal hypersurface may contain a singular set with Hausdorff codimension at least seven. This follows from similar arguments as in \cite{SS81}. To keep this paper less technical, the details will appear in a forthcoming paper.

The classical monotonicity formula plays an important role in the regularity theory for minimal submanifolds, even without the stability assumption. Unfortunately, it ceases to hold once the ball hits the boundary of the minimal submanifold. Therefore, to study the boundary regularity of free boundary minimal submanifolds, we need a monotonicity formula which holds for balls centered at points lying on the constraint submanifold $N$. By an isometric embedding of $M$ into some Euclidean space $\R^L$, we establish a monotonicity formula (Theorem \ref{T:monotonicity}) for free boundary minimal submanifolds relative to \emph{Euclidean} balls of $\R^L$ centered at points on the constraint submanifold $N$.  

We remark that Gr\"{u}ter and Jost proved in \cite{GJ86} a version of monotonicity formula (Theorem 3.1 in \cite{GJ86}) and used it to establish an important Allard-type regularity theorem for varifolds with free boundary. However, 
%they only considered rectifiable varifolds in $\R^{n+1}$ with free boundary (in a weak sense as in \cite[p.133 (7)]{GJ86}) lying on a hypersurface $N \subset \R^{n+1}$. 
the monotonicity formula they obtained \cite[Theorem 3.1]{GJ86} contains an extra term involving the mass of the varifold in a reflected ball, which makes it difficult to apply in some situations (in \cite{LiZ16} for example). 
%Their construction involves certain reflection process across $N$ and thus the resulting monotonicity formula is not so easy to apply in other situations. 
In contrast, our monotonicity formula (Theorem \ref{T:monotonicity}) does not require any reflection which makes it more readily applicable. Moreover, the formula holds in the Riemannian manifold setting for stationary varifolds with free boundary in any dimension and codimension. We expect that our monotonicity formula might be useful in the regularity theory for other natural free boundary problem in calibrated geometries (see for example \cite{Bu03} and \cite{Lo09}). We would like to mention that other monotonicity formulas have been proved for free boundary minimal submanifolds in a Euclidean unit ball (\cite{B12}, \cite{V16}).

%\vspace{1em}
%\sout{Furthermore, if we restrict to dimension $n=2$, we can remove the area upper bound condition by adding an assumption that the surface is embedded.}
Consider now the case of a compact Riemannian $3$-manifold $M$ with boundary $\partial M$, by the remark in the paragraph after Theorem \ref{T:main}, we can assume that $M$ is a compact subdomain of a larger Riemannian manifold $\widetilde{M}$ without boundary and $N=\partial M$ is the constraint submanifold. 
Furthermore, if we assume that the free boundary minimal surface $\Sigma$ is \emph{properly embedded} in $M$ (i.e. $\Sigma \subset M$ and $\Sigma \cap \partial M=\partial \Sigma$), then we prove a stronger uniform curvature estimate similar to the one in Theorem \ref{T:main}, but \emph{independent of the area of $\Sigma$}.

\begin{theorem}
\label{T:curvature estimates 2-D}
%Let $M^3$ be a compact Riemannian 3-manifold with boundary $\partial M$. Then there exists $r_0>0$ small enough (depending only on $M$ and $\partial M$) such that the following holds: let $p\in \partial M$, and suppose $\Si$ is a stable minimal surface in $M$ with free boundary on $\partial M$ such that $p\in\Si$. If we further assume that $\Si$ embedded, then
%$$\max_{0\leq \si\leq r_0}\Big(\si^2\sup_{B_{r_0-\si}(p)}|A^{\Si}|^2\Big) \leq C_0,$$
%where $C_0$ is a constant depending only on the geometry of $B_r(p)$ in $M$.
Let $(M^3, g)$ be a compact Riemannian 3-manifold with boundary $\partial M \neq \emptyset$. Then there exists a constant $C_2>0$ depending only on the geometry of $M$ and $\partial M$, such that if $(\Sigma,\partial \Sigma) \subset (M,\partial M)$ is a compact, properly embedded stable minimal surface with free boundary, then 
\begin{equation*}
\sup_{x\in \Sigma} |A|^2(x)\leq C_2.
\end{equation*}
\end{theorem}

%\textcolor{red}{(The theorem above as stated is not completely precise. The uniform estimates hold only away from the ``Dirichlet boundary'' of $\Sigma$. We need to restate it a bit more carefully.)} 

\begin{remark}
For simplicity, we assume that $\Sigma$ is compact in Theorem \ref{T:curvature estimates 2-D}. This ensures that $\Sigma$ has no boundary points lying in the interior of $M$. Without the compactness assumption, similar uniform estimates still hold as long as we stay away from the points in $\overline{\Sigma} \setminus \Sigma$ inside the interior of $M$ as in Theorem \ref{T:main}.
Note that $\Sigma$ is always locally two-sided under the embeddedness assumption.
\end{remark}

Our proof of Theorem \ref{T:curvature estimates 2-D} involves the theory of minimal laminations which require the minimal surface to be \emph{embedded}.
In view of the celebrated interior curvature estimates for stable \emph{immersed} minimal surfaces in $3$-manifolds by Schoen \cite{Schoen83} (see also \cite{CM11} and \cite{Ros06}), we conjecture that the embeddedness of $\Sigma$ is unnecessary. 
%\sout{condition is not necessary to get a general boundary curvature estimates for stable minimal surfaces with free boundary.}

%\vspace{1em}
\begin{conjecture}
%\sout{For immersed stable minimal surfaces with free boundary in a 3-manifold with boundary, the boundary curvature estimates is always true.}
Theorem \ref{T:curvature estimates 2-D} holds even when $\Sigma$ is immersed.
\end{conjecture}

%\vspace{1em}
The organization of the paper is as follows. In section \ref{S:definitions}, we give the basic definitions for free boundary minimal submanifolds in any dimension and codimension and discuss the notion of stability in the hypersurface case. In section \ref{S:monotonicity}, we prove the monotonicity formula (Theorem \ref{T:monotonicity}) for stationary varifolds with free boundary near the free boundary in any dimension and codimension. 
%The same formula also holds in the context of varifolds (Theorem \ref{T:monotonicity-varifold}). 
In section \ref{S:curvature-estimates}, we prove our main curvature estimates (Theorem \ref{T:main-curvature-estimates}) for stable free boundary minimal hypersurfaces near the free boundary. In section \ref{S:2d-curvature-estimates}, we prove 
%\sout{the 2-dimensional curvature estimates for embedded stable free boundary minimal surfaces up to the free boundary} 
the stronger curvature estimate (Theorem \ref{T:curvature estimates 2-D}) in the case of \emph{properly embedded} stable free boundary minimal surfaces in a Riemannian $3$-manifold with boundary. In section \ref{S:convergence}, we prove a general convergence result for free boundary minimal submanifolds (in any dimension and codimension) satisfying uniform bounds on area and the second fundamental form. Finally, in section \ref{S:lamination convergence}, we prove a lamination convergence result for free boundary minimal surfaces in a three-manifold with uniform bound only on the second fundamental form of the minimal surfaces.

\vspace{0.5em}
{\bf Acknowledgements}: 
The authors would like to thank Prof. Richard Schoen for his continuous encouragement. They also want to thank Prof. Shing Tung Yau, Prof. Tobias Colding and Prof. Bill Minicozzi for their interest in this work.
% people to acknowledge by Martin
% people to acknowledge by Xin
M. Li is partially supported by a research grant from the Research Grants Council of the Hong Kong Special Administrative Region, China [Project No.: CUHK 24305115] and CUHK Direct Grant [Project Code: 4053118].
X. Zhou is partially supported by NSF grant DMS-1406337.
The authors are grateful for the anonymous referee for valuable comments.
%\textcolor{red}{Acknowledgement of Q. Guang.}

%%%%%%%%%%%%%%%%%%%%%%%%%%%%%%%%%%
% Section 2    	  Free Boundary Minimal Submanifolds	             %
%%%%%%%%%%%%%%%%%%%%%%%%%%%%%%%%%%
\section{Free Boundary Minimal Submanifolds}
\label{S:definitions}

In this section, we give the definition of \emph{free boundary minimal submanifolds} (Definition \ref{D:freeboundary}) and the notion of \emph{stability} (Definition \ref{D:stable}) in the hypersurface case. We also prove a reflection principle (Lemma \ref{L:reflection}) which will be useful in subsequent sections.

Let $(M,g)$ be an $m$-dimensional Riemannian manifold, and $N \subset M$ be an embedded $n$-dimensional constraint submanifold. We will always assume $M,N$ are smooth without boundary unless otherwise stated. Suppose $\Sigma$ is a $k$-dimensional smooth manifold with boundary $\partial \Sigma$ (possibly empty).

\begin{definition}
\label{D:immersion}
We use $(\Sigma,\partial \Sigma) \looparrowright (M,N)$ to denote an immersion $\varphi:\Sigma \to M$ such that $\varphi(\partial \Sigma) \subset N$. If, furthermore, $\varphi$ is an embedding, we denote it as $(\Sigma, \partial \Sigma) \hookrightarrow (M,N)$. 
%\textcolor{red}{\sout{We often say that $(\Sigma,\partial \Sigma) \subset (M,N)$ is an \emph{immersed} (respectively \emph{embedded}) submanifold without explicitly mentioning the immersion (or embedding) $\varphi$.}}  
An embedded submanifold $(\Sigma,\partial \Sigma) \subset (M,N)$ is said to be \emph{proper} if $\varphi(\Sigma) \cap N = \varphi(\partial \Sigma)$.
\end{definition}

%Let $M^m$ be a Riemannian $m$-manifold and $N^n \subset M$ be an embedded $n$-submanifold. The manifolds $M$ and $N$ are assumed to be smooth. Let $\varphi:\Si^k \to M$ be an immersed $k$-dimensional submanifold with $\varphi(\partial\Si) \subset N$. For simplicity, we omit $\varphi$ and write $(\Sigma,\partial \Sigma) \subset (M,N)$.

\begin{definition}
\label{D:freeboundary}
We say that $(\Sigma,\partial \Sigma) \subset (M,N)$ is an immersed (resp. embedded) \emph{free boundary minimal submanifold} if
\begin{itemize}
\item[(i)] $\varphi: \Sigma \to M$ is a minimal immersion (resp. embedding), and
%$\Si \subset M$ is an immersed (resp. embedded) minimal submanifold, and
\item[(ii)] $\Sigma$ meets $N$ orthogonally along $\partial \Si$.
\end{itemize}
\end{definition}

\begin{remark}
Condition (ii), is often called the \emph{free boundary condition}. 
%is understood as follows. We equip $\Sigma$ with the induced metric from the immersion $\varphi$ and denote $\nu_{\partial \Sigma}$ to be the outward unit normal of $\partial \Sigma$ relative to $\Sigma$. We say that $\Sigma$ meets $N$ orthogonally along $\partial \Si$ if the outward unit co-normal $\eta_\Sigma:=d\varphi (\nu_{\partial \Sigma})$ is orthogonal to $N$ everywhere along $\partial \Sigma$. %(see Figure \ref{neighborhood near boundary}). 
Note that both conditions (i) and (ii) are local properties.
\end{remark}

%above means that if $\eta_\Sigma=d\varphi(\nu_{\partial \Sigma})$ is the outward unit co-normal of $\partial \Sigma$ (\textcolor{blue}{where $\nu_{\partial \Si}$ is the outward unit normal of $\partial \Si$ relative to $\Sigma$ with the induced metric from the immersion $\varphi:\Sigma \to M$}), then $\eta_\Si \perp N$ (see Figure \ref{neighborhood near boundary}). Note that conditions (i) and (ii) are both local properties.

\begin{comment}
\begin{figure}[h]
\centering
\includegraphics[height=1.5in]{neighborhood_near_boundary_curvature_estimates.png}
\caption{The free boundary condition}
\label{neighborhood near boundary}
\end{figure}
\end{comment}

Free boundary minimal submanifolds can be characterized variationally as critical points to the $k$-dimensional area functional of $(M,g)$ among the class of all immersed $k$-submanifolds $(\Sigma,\partial \Sigma) \subset (M,N)$. Given a smooth $1$-parameter family of immersions $\varphi_t:(\Sigma,\partial \Sigma) \to (M,N)$, $t \in (-\epsilon,\epsilon)$, whose variation vector field $X(x)=\left. \frac{d}{dt}\right|_{t=0} \varphi_t(x)$ is compactly supported in $\Sigma$, the \emph{first variational formula} (c.f. \cite[\S 1.13]{CM11}) says that 
%\begin{eqnarray}
\begin{equation}
\label{first variation formula0}
\left. \frac{d}{dt}\right|_{t=0} \Area (\varphi_t(\Si)) = \int_{\Si} \Div_{\Si} X \; da 
= -\int_{\Si} X \cdot H \; da+\int_{\partial\Si} X \cdot \eta \; ds,
%\end{eqnarray}
\end{equation}
where $H$ is the mean curvature vector of the immersion $\varphi_0:\Sigma \to M$ with outward unit conormal $\eta$, $da$ and $ds$ are the induced measures on $\Sigma$ and $\partial \Sigma$ respectively. 
Since $\varphi_t(\partial \Sigma) \subset N$ for all $t$, the variation vector field $X$ must be tangent to $N$ along $\partial \Sigma$. Therefore, $\varphi:(\Sigma,\partial \Sigma) \looparrowright (M,N)$ is a free boundary minimal submanifold if and only if (\ref{first variation formula0}) vanishes for all compactly supported variational vector field $X$ with $X(p) \in T_pN$ for all $p \in \partial \Sigma$, which is equivalent to conditions (i) and (ii) in Definition \ref{D:freeboundary}.
%deformation $\varphi_t:(\Sigma,\partial \Sigma) \to (M,N)$ with $\varphi_0=\varphi$.

Since free boundary minimal submanifolds are critical points to the area functional, we can look at the second variation and study their stability. Roughly speaking, a free boundary minimal submanifold is said to be \emph{stable} if the second variation is non-negative. For simplicity and our purpose, we will only consider the \emph{hypersurface} case, i.e. $\dim \Sigma =\dim N =\dim M-1$.
Recall that an immersion $\varphi:\Sigma \to M$ is said to be \emph{two-sided} if there exists a globally defined continuous unit normal vector field $\nu$ on $\Sigma$.

\begin{definition}
\label{D:stable}
An immersed free boundary minimal hypersurface $\varphi:(\Sigma,\partial \Sigma) \looparrowright (M,N)$ is said to be \emph{stable} if it is two-sided and satisfies the stability inequality, i.e.
\begin{eqnarray}
\label{E:stability}
0 & \leq & \left. \frac{d^2}{dt^2}\right|_{t=0} \Area (\varphi_t(\Si))   \nonumber \\ 
                                           & =& \int_{\Si} |\nabla_{\Si} f|^2- (|A^{\Si}|^2+\Ric(\nu, \nu)) f^2 \; da -\int_{\partial\Si} A^N(\nu, \nu) f^2 \; ds, 
\end{eqnarray}
where $\varphi_t:(\Sigma,\partial \Sigma) \looparrowright (M,N)$ is any compactly supported variation of $\varphi_0=\varphi$ with variation field $X=f \nu$, 
$A^{\Si}$ and $A^N$ are the second fundamental forms of $\Si$ and $N$ in $M$ respectively, and $\Ric$ is the Ricci curvature of $M$.
\end{definition}

\begin{remark}
The sign convention of $A^N$ in (\ref{E:stability}) is taken such that $A^N \geq 0$ if $N=\partial \Omega$ is the boundary of a convex domain in $M$.
\end{remark}

One particularly important example is $M=\R^{n+1}$ and $N=\R^n=\{x_1=0\}$. Let $\R^{n+1}_+=\{x_1 \geq 0\}$ and $\theta:\R^{n+1} \to \R^{n+1}$ be the reflection map across $\R^n$. We have the following reflection principle that relates free boundary minimal hypersurfaces with minimal hypersurfaces without boundary.

\begin{lemma}[Reflection principle]
\label{L:reflection}
If $(\Sigma,\partial \Si) \looparrowright (\R^{n+1},\R^n)$ is an immersed stable free boundary minimal hypersurface, then $\Si \cup \theta(\Si)$ is an immersed stable minimal hypersurface (without boundary) in $\R^{n+1}$.
\end{lemma}

\begin{proof}
Since minimality is preserved under the isometry $\theta$ of $\R^{n+1}$ and that $\Sigma$ is orthogonal to $\R^n$ along $\partial \Sigma$, $\Sigma \cup \theta(\Si)$ is a $C^1$ minimal hypersurface in $\R^{n+1}$ without boundary. Higher regularity for minimal hypersurfaces implies that it is indeed smooth across $\partial \Sigma$. Stability follows directly from the definition since the boundary term in (\ref{E:stability}) vanishes for $N=\R^n$.
\end{proof}

%%%%%%%%%%%%%%%%%%%%%%%%%%%%%%%%%%
% Section 3    	  Monotonicity formula          		     %
%%%%%%%%%%%%%%%%%%%%%%%%%%%%%%%%%%

\section{Monotonicity formula}
\label{S:monotonicity}

%In this section, we prove a monotonicity formula (Theorem \ref{T:monotonicity}) for free boundary minimal submanifolds in Riemannian manifolds for any dimension and codimension. We show that the same formula (Theorem \ref{T:monotonicity-varifold}) holds for stationary varifolds with free boundary (c.f. Definition \ref{D:freebdy}), which generalizes the notion of free boundary minimal submanifolds to allow singularities.

In this section, we prove a monotonicity formula (Theorem \ref{T:monotonicity}) for stationary varifolds with free boundary (c.f. Definition \ref{D:freebdy}) in Riemannian manifolds for any dimension and codimension. The monotonicity formula for free boundary minimal submanifolds is then a direct corollary.

Throughout this section, we will consider $M \subset \R^L$ as an embedded $m$-dimensional submanifold (by Nash isometric embedding theorem) and a compact closed $n$-dimensional constraint submanifold $N \subset M$. 
%Let $M^m$ be a closed Riemannian $m$-manifold which is isometrically embedded into some Euclidean space $\mathbb{R}^L$ by Nash embedding theorem. Fix a \textit{constraint submanifold} $N^n \subset M^m$ which is a compact embedded $n$-submanifold ($1 \leq n  \leq m$) possibly with boundary. 
We will denote $\tB(p, r)$ to be the open Euclidean ball in $\R^L$ with center $p$ and radius $r>0$. The second fundamental form of $M$ in $\R^L$ is denoted by $A^M$. 
%The Euclidean norm in $\mathbb{R}^L$ will be denoted by $| \cdot |$. Let $\Hom(\R^L,\R^L)$ denote the space of linear operators on $\R^L$ with the usual operator norm. 
%\[ \textcolor{red}{\sout{\|Q\|:=\sup\{ |Q(v)| : v \in \R^L, \; |v|=1\}.}} \]

%%%%%%%%%%%%%%%%%%%%%%%%%%%%%%%%%%%%%%%%%%
\begin{comment}
We begin with a discussion on the first variation formula for submanifolds in $M$ with respect to deformations \emph{in $\R^L$} that do not necessarily preserve $M$. Let $X$ be a compactly supported smooth vector field in $\R^L$ which generates a one-parameter family of diffeomorphisms $\{\phi_t\}_{t \in \R}$ of $\R^L$. Assume that $(\Si,\partial \Si) \looparrowright (M,N)$ is an immersed free boundary minimal submanifold, then the first variation formula (c.f. \cite[\S 9]{Si83}) relative to the variation vector field $X$ gives
\begin{equation}
\label{E:first-variation}
\left. \frac{d}{dt}\right|_{t=0} \Area (\phi_t(\Si))=\int_{\Si} \Div_{\Si}X \;da=\int_{\Si} X \cdot  \tr_{\Si} A^M \; da +\int_{\partial\Si} X \cdot \eta \; ds,
\end{equation}
where $\tr_{\Si} A^M=\sum_{i=1}^k A^M(e_i, e_i)$ for an orthonormal basis $\{e_1, \cdots, e_k\}$ of $T_x\Si$ and $\eta$ is the outward unit co-normal of $\partial \Si$.
\end{comment} 
%%%%%%%%%%%%%%%%%%%%%%%%%%%%%%%%%%%%%%%%%%

We begin with a discussion on the notion of \emph{stationary varifolds with free boundary}. Let $\V_k(M)$ denote the closure (with respect to the weak topology) of rectifiable \emph{$k$-varifolds} in $\mathbb{R}^L$ which is supported in $M$ (c.f. \cite[2.1(18)(g)]{P81}). As usual, the \emph{weight} of a varifold $V \in \V_k(M)$ is denoted by $\|V\|$. 
%which is a Radon measure on $\mathbb{R}^L$ supported on $M$. 
We refer the readers to the standard reference \cite{Si83} on varifolds.

%Let $N^n \subset M^m$ be the constraint submanifold as before. As $N \subset M \subset \R^L$, 
We use $\mathfrak{X}(M,N)$ to denote the space of smooth vector fields $X$ compactly supported on $\mathbb{R}^L$ such that $X(x) \in T_xM$ for all $x \in M$ and $X(p)\in T_p N$ for all $p \in N$. Any such vector field $X \in \mathfrak{X} (M, N)$ generates a one-parameter family of diffeomorphisms $\phi_t: M \to M$ with $\phi_t(N)=N$ and the first variation of a varifold $V \in \V_k(M)$ along $X$ is defined by
\[ \delta V (X) := \left. \frac{d}{dt}\right|_{t=0} \|(\phi_t)_\sharp V\|(M),\]
where $(\phi_t)_\sharp V \in \V_k(M)$ is the pushforward of $V$ by the diffeomorphism $\phi_t$ (c.f. \cite[2.1(18)(h)]{P81}).

\begin{definition}
\label{D:freebdy}
 A $k$-varifold $V\in\V_k(M)$ is said to be \emph{stationary with free boundary on $N$} if $\delta V(X) =0$ for all $X \in \mathfrak{X}(M, N)$. 
\end{definition}

This generalizes the notion of free boundary minimal submanifolds to allow singularities. By the first variation formula for varifolds \cite[39.2]{Si83}, a $k$-varifold $V \in \V_k(M)$ is stationary with free boundary on $N$ if and only 
\begin{equation}
\label{E:stationary}
\int \Div_S X(x) \; dV(x, S)=0
\end{equation}
for all $X \in \mathfrak{X}(M,N)$. If $X$ is not tangent to $M$ but $X(p) \in T_pN$ for all $p \in N$, then (\ref{E:stationary}) implies that
\begin{equation}
\label{E:first-variation-X-varifold}
\int \Div_S X(x) \; dV(x,S) = \int X(x) \cdot \tr_S A^M \; dV(x,S),
\end{equation}
where $S\subset T_x M$ is an arbitrary $k$-plane, and $\tr_S A^M=\sum_{i=1}^k A^M(e_i, e_i)$ for an orthonormal basis $\{e_1, \cdots, e_k\}$ of $S$.
%It is then easy to see that the proof of Theorem \ref{T:monotonicity} works without any change for stationary varifolds with free boundary using (\ref{E:first-variation-X-varifold}) in place of (\ref{E:first-variation-X}).

\vspace{0.5em}
The key idea to derive our monotonicity formula near a base point $p\in N$ is to find a special test vector field $X$ which is asymptotic (near $p$) to the radial vector field centered at $p$ and, at the same time, tangential along the constraint submanifold $N$. Our choice of $X$ is largely motivated by \cite{Al75, GJ86}, and we add the following preliminary results for completeness.
%To find such a vector field $X$ with desired estimates, we will need the following preliminary results.

Let us review some local geometry of the $k$-dimensional compact closed constraint submanifold $N$ in $\R^L$ essentially following the discussions in \cite[\S 2]{Al75}.
%Note that $\partial N$ can be non-empty so that we can localize our results in this section. We will, however, only consider the case $\partial N=\emptyset$ in subsequent sections. 
We always identify a linear subspace $P \subset \R^L$ with its orthogonal projection $P \in \Hom(\R^L,\R^L)$ onto this subspace. Using this notion, we define the maps $\tau,\nu:N \to \Hom(\mathbb{R}^L,\mathbb{R}^L)$ to be
\[ \tau(p):=T_{p}N \quad \text{ and } \quad \nu(p):=(T_p N)^{\perp},\]
where $T_p N$ is the tangent space of $N$ in $\R^L$, and $(T_p N)^{\perp}$ is the orthogonal complement of $T_p N$ in $\R^L$. 
%(see Figure \ref{F:tau and nu}). 

%%%%%%%%%%%%%%%%%%%%%%%%%%%%%%%%%%%
\begin{comment}
\begin{figure}[h]
\centering
\includegraphics[height=1.5in]{tau_and_nu.png}
\caption{$\tau$ and $\nu$}
\label{F:tau and nu}
\end{figure}
\end{comment}
%%%%%%%%%%%%%%%%%%%%%%%%%%%%%%%%%%%

To bound the turning of $N$ inside $\mathbb{R}^L$,  we define as in \cite{Al75} a global geometric quantity
\[ \ka:=\inf \left\{ t \geq 0 : |\nu(x)(y-x)|\leq\frac{t}{2}\; |y-x|^2 \text{ for all $x, y \in N$} \right\}. \] 
By the compactness and smoothness of $N$, $\ka \in [0,\infty)$ and thus one can define the \emph{radius of curvature for $N$} to be
\begin{equation}
\label{D:R}
R_0:=\ka^{-1} \in (0,\infty].
\end{equation}
Let $\xi$ be the nearest point projection map onto $N$ and $\rho(\cdot):=\dist_{\mathbb{R}^L}(\cdot,N)$ be the distance function to $N$ \emph{in $\mathbb{R}^L$}, both defined on a tubular neighborhood of $N$. More precisely, if we define 
%for each $p \in N \setminus \partial N$ the number
%\begin{equation}
%\label{E:rp}
%r(p):= \frac{R_0 d}{R_0+d} \in (0,\min(R_0,d) \,], 
%\end{equation}
%where $d=\dist_{\mathbb{R}^L}(p,\partial N)$ (Note that if $\partial N = \emptyset$, we take $d=\infty$ and $r(p) \equiv R_0$), and 
the open set 
\[ A:= \bigcup_{p \in N} \tB(p, R_0) \]
which is an open neighborhood of $N$ inside $\R^L$, we have the following lemma from \cite[Lemma 2.2]{Al75}.

\begin{lemma}
\label{L:local-geometry}
With the definitions as above, $\xi$, $\rho$, $\tau$, $\nu$ are well-defined and smooth on $A$. Moreover, we have the following estimates:
\begin{equation}
\label{E:nu-estimate}
\|D\nu_p (v)\|\leq \ka|v|, \quad \forall p\in N, v\in T_pN,
\end{equation}
\begin{equation}
\label{E:xi-estimate-a}
\|D\xi_a\|\leq \frac{1}{1-\ka\rho(a)},\quad \forall a\in A,
\end{equation}
\begin{equation}
\label{E:xi-estimate-b}
|\xi(a)-p|\leq \frac{|a-p|}{1-\ka |a-p|}, \quad \forall p \in N, a \in \tB(p, R_0).
\end{equation}
\end{lemma}

\begin{proof}
See \cite[Lemma 2.2]{Al75}.
\end{proof}

From now on, we fix a point $p \in N$. Without loss of generality, we can assume that $p=0$ after a translation in $\mathbb{R}^L$. 
%Let $r_0:=r(0)$ in (\ref{E:rp}). 
By Lemma \ref{L:local-geometry}, we can define a smooth map $\zeta:  \tB(0, R_0) \to \mathbb{R}^L$ by 
\begin{equation}
\label{D:zeta}
\zeta(x):=-\nu\big(\xi(x)\big)\xi(x).
\end{equation}
Note that $-\zeta(x)$ is the normal component (with respect to $T_{\xi(x)}N$) of the vector $\xi(x)-p$ (which is equal to $\xi(x)$ when $p=0$). See Figure \ref{zeta}.

\begin{figure}[h]
\centering
\includegraphics[height=0.9in]{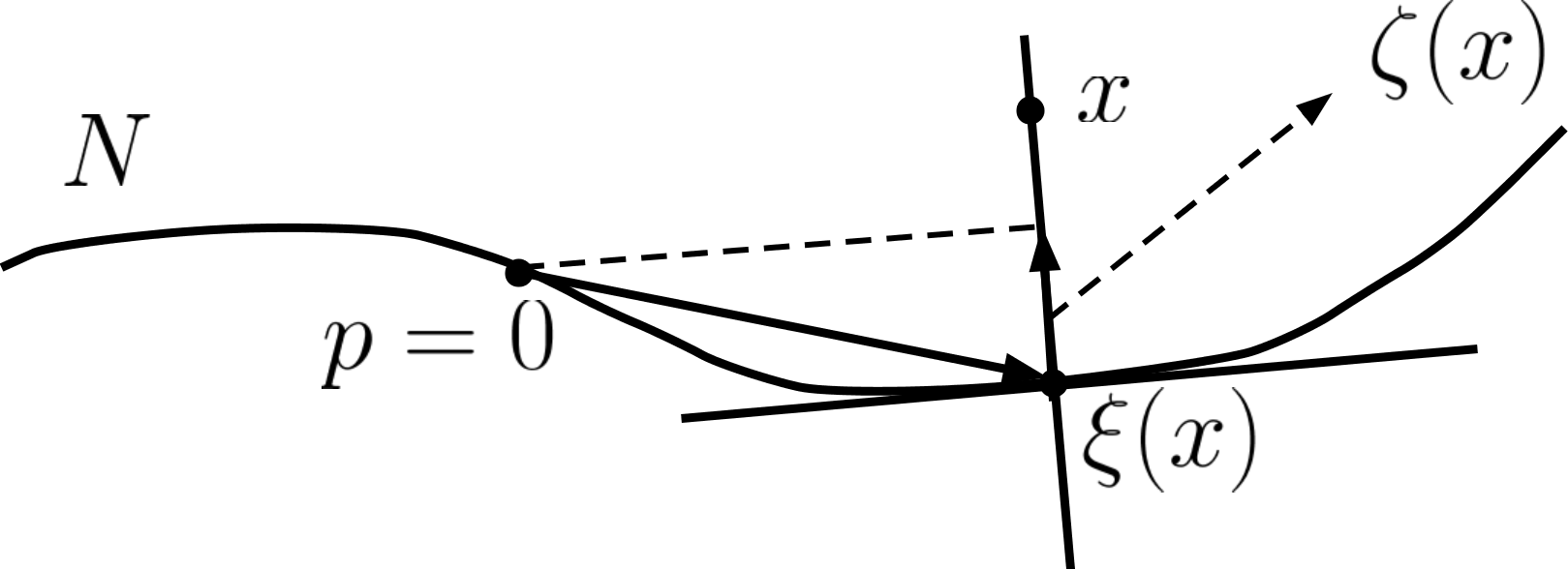}
\caption{Definition of $\zeta$}
\label{zeta}
\end{figure}

\begin{lemma}
\label{L:zeta}
Fix any $s \in (0, R_0)$, if we let $\ga=\frac{R_0}{2(R_0-s)^2}$, then
\begin{equation}
\label{E:zeta-estimate}
\|D\zeta_x\|\leq 2\ga |x| \quad \text{and} \quad |\zeta(x)|\leq \ga|x|^2 \quad \forall x \in  \tB(0, s).
\end{equation}
\end{lemma}

\begin{proof}
Fix $s \in (0, R_0)$ and any $x \in \tB(0,s)$. As $D\xi_x(v)\in T_{\xi(x)} N$ for any $v\in\R^L$, we have $\nu(\xi(x)) D\xi_x(v)=0$ for any $v$, thus
\begin{displaymath}
\begin{split}
D\zeta_x(v) & =- [D\nu_{\xi(x)} \circ D\xi_x (v)] (\xi(x)) - \nu(\xi(x)) D\xi_x(v) \\
&= - [D\nu_{\xi(x)} \circ D\xi_x (v)] (\xi(x)).
\end{split}
\end{displaymath}
Therefore, we have by (\ref{E:nu-estimate}), (\ref{E:xi-estimate-a}), (\ref{E:xi-estimate-b}), $\rho(x) \leq |x|$ and $D\xi_x(v)\in T_{\xi(x)} N$,
\[ \|D\zeta_x\| \leq \ka \cdot \frac{1}{1-\ka\rho(x)}\cdot  \frac{|x|}{1-\ka|x|}\leq \frac{R_0}{(R_0-|x|)^2}|x| \leq 2 \ga |x|.\]
The estimate for $|\zeta(x)|$ follows from a line integration from $x=0$ using that $\zeta(0)=0$.
\end{proof}

We can now state our monotonicity formula.

\begin{theorem}[Monotonicity Formula]
\label{T:monotonicity}
Let $M$ be an embedded $m$-dimensional submanifold in $\R^L$ with second fundamental form $A^M$ bounded by some constant $\La>0$, i.e. $|A^M|\leq\La$. Suppose $N \subset M$ is a compact, closed, embedded $n$-dimensional submanifold, and $V\in\V_k(M)$ is a stationary $k$-varifold with free boundary on $N$.
%$(\Si,\partial \Si) \looparrowright (M,N)$ is a $k$-dimensional immersed free boundary minimal submanifold. 

For any $p \in N$ and $0 < \si < \rho < \frac{1}{2} R_0$ as defined in (\ref{D:R}), we have
%\[ e^{\La_1\si}\frac{\Area(\Si\cap \tB(p, \si))}{\si^k}\leq e^{\La_1\rho}\frac{\Area(\Si\cap \tB(p, \rho))}{\rho^k}-\int_{\Si\cap \tA(p, \si, \rho)}\frac{e^{\La_1 r}|\nabla_{\Si}^{\perp}r|^2}{(1+\ga r)r^k}, \]
\[ e^{\La_1\si} \frac{\|V\|(\tB(p, \si))}{\si^k}
                    \leq e^{\La_1\rho}\frac{\|V\|(\tB(p, \rho))}{\rho^k}-\int_{G_k(\tA(p, \si, \rho))}\frac{e^{\La_1 r}|\nabla_{S}^{\perp}r|^2}{(1+\ga r)r^k}dV(x, S). \]
Here $\ga=\frac{2}{R_0}$ is defined in Lemma \ref{L:zeta} (with $s=\frac{1}{2} R_0$), $\La_1:=k(\La+3\ga)$, $r(x):=|x-p|$, $\nabla_S^\perp r$ is the projection of $\nabla r$ to the orthogonal complement $S^\perp$ of the $k$-plane $S \subset \mathbb{R}^L$, and $G_k(\tA(p, \si, \rho)):=\tA(p, \si, \rho) \times G(L,k)$ is the restriction of the $k$-dimensional Grassmannian on $\mathbb{R}^L$ restricted to $\tA(p, \si, \rho):=\tB(p, \rho)\setminus \tB(p, \si)$.
\end{theorem}

\begin{proof}
As before, we can assume $p=0$ by a translation in $\mathbb{R}^L$. The monotonicity formula will be obtained by choosing a suitable test vector field $X$ in (\ref{E:first-variation-X-varifold}). Define 
\[ X(x):=\varphi(r)\big(x+\zeta(x)\big), \]
where $r=|x|$ and $\varphi \geq 0$ is a smooth cutoff function with $\varphi^{\pr}\leq 0$, and $\varphi(r)=0$ for $r\geq \frac{1}{2} R_0$. When $x\in N$, we have $\xi(x)=x$ and thus 
\[ x+\zeta(x)=x-\nu(x)x=\tau(x) x \in T_xN.\]
Hence $X(x)\in T_xN$ for all $x \in N$, and (\ref{E:first-variation-X-varifold}) holds true for such $X$.
%The second term in the first variation formula (\ref{E:first-variation}) drops out and thus we have
%\begin{equation}
%\label{E:first-variation-X}
%\int_{\Si} \Div_{\Si}X \;da=\int_{\Si} X \cdot \tr_{\Si}A^M \; da.
%\end{equation}

For any $k$-dimensional subspace $S \subset \mathbb{R}^L$, by the definition of $X$,
\begin{displaymath}
\begin{split}
\Div_S X(x) & =\varphi(r)\big(\Div_S x+\Div_S\zeta(x)\big)+\varphi^{\pr}(r)\nabla^S r\cdot\big(x+\zeta(x)\big)\\
                      & =\varphi(r)(k+\Div_S\zeta(x))+\varphi^{\pr}(r)\big[r(1-|\nabla_S^{\perp}r|^2)+\nabla^S r\cdot\zeta(x)\big].
\end{split}
\end{displaymath}
%where $\nabla_S^\perp r$ is the component of $\nabla r$ along $S^\perp \subset \R^L$. 
By (\ref{E:zeta-estimate}), we have the estimates
\[ |\Div_S\zeta(x)|\leq k \|D\zeta_x \| \leq 2k\ga r,\]
\[ |\nabla^S r\cdot\zeta(x)|\leq|\zeta(x)|\leq \ga r^2.\]
Using the fact that $\varphi\geq 0$ and $\varphi^{\pr}\leq 0$, we have the following estimates
$$\Div_S X(x)\geq \varphi(r)(k-2k\ga r)+\varphi^{\pr}(r)\big[r(1-|\nabla_S^{\perp}r|^2)+\ga r^2\big],$$
$$|X(x)|\leq\varphi(r)(|x|+|\zeta(x)|)\leq \varphi(r)(r+\ga r^2).$$
Plugging these estimates into (\ref{E:first-variation-X-varifold}) and using the bound $|A^M|\leq\La$,
\begin{displaymath}
\begin{split}
&\int\varphi^{\pr}(r)r(1+\ga r) d\|V\|  +k\int\varphi(r)d\|V\| \\
                                                    & \leq \int \varphi^{\pr}(r)r|\nabla_S^{\perp}r|^2 dV(x, S)+ k\La\int_{\Si}\varphi(r)r(1+\ga r)d\|V\|+2k\ga\int_{\Si}\varphi(r)r d\|V\|.
\end{split}
\end{displaymath}
Fix a smooth cutoff function $\phi:[0, \infty)\rightarrow[0, 1]$ such that $\phi^{\pr}\leq 0$ and $\phi(s)=0$ for $s\geq 1$. For any $\rho \in (0,\frac{1}{2}R_0)$, if we define $\varphi(r)=\phi(\frac{r}{\rho})$, then it is a cutoff function satisfying all the assumptions above. Moreover, $r\varphi^{\pr}(r)=-\rho\frac{d}{d\rho}\varphi(\frac{r}{\rho})$. Plugging into the inequality above, using the fact that $\phi(\frac{r}{\rho})=0$ for $r \geq \rho$,
\begin{displaymath}
\begin{split}
-\rho(& 1+  \ga\rho)  \frac{d}{d\rho}\int \phi \left(\frac{r}{\rho}\right) +k\int \phi  \left(\frac{r}{\rho}\right) \\
                          & \leq -\rho\frac{d}{d\rho}\int \phi  \left(\frac{r}{\rho}\right)|\nabla_S^{\perp}r|^2+ k\La\rho(1+\ga\rho)\int \phi  \left(\frac{r}{\rho}\right)+2k\ga\rho\int \phi  \left(\frac{r}{\rho}\right).
\end{split}
\end{displaymath}
Adding $k\ga\rho\int \phi(\frac{r}{\rho})$ to both sides of the inequality, we obtain
\begin{displaymath}
\begin{split}
-\rho(1+\ga\rho)  \frac{d}{d\rho} & \int \phi  \left(\frac{r}{\rho}\right) +k(1+\ga\rho)\int \phi  \left(\frac{r}{\rho}\right) \\
                          & \leq -\rho\frac{d}{d\rho}\int \phi  \left(\frac{r}{\rho}\right) |\nabla_S^{\perp}r|^2+ k \rho [ \La (1+\ga\rho)+3 \ga] \int \phi  \left(\frac{r}{\rho}\right).
\end{split}
\end{displaymath}
Denote $I(\rho)=\int \phi(\frac{r}{\rho})d\|V\|$ and $J(\rho)=\int \phi(\frac{r}{\rho})|\nabla_S^{\perp}r|^2 dV(x, S)$, then we have
$$(1+\ga\rho)\frac{d}{d\rho}\left( \frac{I(\rho)}{\rho^k}\right) \geq \frac{J^{\pr}(\rho)}{\rho^k}-k[\La(1+\ga\rho)+3\ga]\frac{I(\rho)}{\rho^k},$$
which clearly implies
$$\frac{d}{d\rho}\left( \frac{I(\rho)}{\rho^k}\right) + k(\La+3\ga) \frac{I(\rho)}{\rho^k} \geq \frac{J^{\pr}(\rho)}{(1+\ga \rho)\rho^k}.$$
Therefore, we can rewrite it into the form
$$\frac{d}{d\rho}\left( e^{k(\La+3\ga)\rho}\frac{I(\rho)}{\rho^k}\right) \geq \frac{e^{k(\La+3\ga)\rho}}{(1+\ga\rho)\rho^k}J^{\pr}(\rho).$$
The monotonicity formula follows by letting $\phi$ approach the characteristic function of $[0, 1]$.
\end{proof}

%%%%%%%%%%%%%%%%%%%%%%%%%%%%%%%%%%%
\begin{comment}
Therefore, we have the following theorem.

\begin{theorem}[Monotonicity formula for varifolds]
\label{T:monotonicity-varifold}
Let $V\in\V_k(M)$ be a stationary $k$-varifold with free boundary on $N$. Using the same notations as in Theorem \ref{T:monotonicity}, we have
\[ e^{\La_1\si} \frac{\|V\|(\tB(p, \si))}{\si^k}
                    \leq e^{\La_1\rho}\frac{\|V\|(\tB(p, \rho))}{\rho^k}-\int_{G_k(\tA(p, \si, \rho))}\frac{e^{\La_1 r}|\nabla_{S}^{\perp}r|^2}{(1+\ga r)r^k}dV(x, S). \]
Here, $\nabla_S^\perp r$ is the projection of $\nabla r$ to the orthogonal complement $S^\perp$ of the $k$-dimensional subspace $S \subset \mathbb{R}^L$, and $G_k(\tA(p, \si, \rho)):=\tA(p, \si, \rho) \times G(L,k)$ is the restriction of the $k$-dimensional Grassmannian on $\mathbb{R}^L$ restricted to $\tA(p, \si, \rho)$.
\end{theorem}
\end{comment}
%%%%%%%%%%%%%%%%%%%%%%%%%%%%%%%%%%%

%%%%%%%%%%%%%%%%%%%%%%%%%%%%%%%%%%
% Section 4    	 	  Curvature estimates                      		     %
%%%%%%%%%%%%%%%%%%%%%%%%%%%%%%%%%%
\section{Curvature estimates}
\label{S:curvature-estimates}

In this section, we prove our main curvature estimates (Theorem \ref{T:main-curvature-estimates}) which imply Theorem \ref{T:main}. The estimates hold for immersed stable free boundary minimal hypersurfaces in any closed Riemannian manifold $(M,g)$ with constraint hypersurface $N \subset M$. Moreover, the estimates are local and \emph{uniform} in the sense that the constants only depend on the geometry of $M$ and $N$, and the area of the minimal hypersurface. Throughout this section, we will assume that the $(n+1)$-dimensional closed Riemannian manifold $(M^{n+1},g)$ is isometrically embedded into $\R^L$ and $N \subset M$ is a compact embedded hypersurface in $M$ with $\partial N=\emptyset$.
%consider \emph{immersed} free boundary minimal hypersurfaces which minimize area up to second order and derive local uniform curvature estimate near the free boundary, subject to a uniform area bound. 

%As in section \ref{S:monotonicity}, let $M^m$ be a closed Riemannian $m$-manifold which is isometrically embedded into $\mathbb{R}^L$ and $N^n \subset M^m$ be a compact embedded submanifold. Throughout this section we assume furthermore that $n+1=m$ and $\partial N=\emptyset$ (i.e. $N$ is a \emph{closed hypersurface} in $M$).

%%%%%%%%%%%%%%%%%%%%%%%%%%%%%%%%%%%%%%%%%%
\begin{comment}
Since $M$ is compact, there exists a constant $c_M>1$ (depending only on the isometric embedding of $M$ in $\mathbb{R}^L$) such that for all $x,y \in M$,
\[ \dist_{\R^L}(x, y)\leq \dist_M(x, y)\leq c_M \dist_{\R^L}(x, y), \]
which implies that for all $p\in M$ and $r>0$,
\begin{equation}
\label{E:ball-comparison}
\tB(p, c_M^{-1}r) \cap M \subset B(p, r)\subset \tB(p, r) \cap M, 
\end{equation}
where $B(p,r) \subset M$ is the open geodesic ball of $M$ centered at $p$ with radius $r>0$, and $\tB(p,r)$ is the open Euclidean ball of radius $r>0$ centered at $p$.
\end{comment}
%%%%%%%%%%%%%%%%%%%%%%%%%%%%%%%%%%%%%%%%%%

Denote $B(p,r) \subset M$ as the open geodesic ball of $M$ centered at $p$ with radius $r>0$. Since the intrinsic distance on $M$ and the extrinsic distance on $\mathbb{R}^L$ are equivalent near a given point $p\in M$, we can WLOG assume that the monotonicity formula (Theorem \ref{T:monotonicity}) holds true for geodesic balls when the radius is less than some $R_0>0$ (depending only on $(M, N)$ and the embedding to $\R^L$). Now we can state our main curvature estimates near the boundary.

\begin{theorem}
\label{T:main-curvature-estimates}
Let $2 \leq n \leq 6$. %Suppose $M^{n+1} \subset \R^L$ is a closed $(n+1)$-dimensional submanifold embedded in $\mathbb{R}^L$, satisfying (\ref{E:ball-comparison}),  and let $N \subset M$ be a closed embedded hypersurface with $R_0>0$ as defined in (\ref{D:R}). 
Suppose $M^{n+1} \subset \R^L, N, R_0$ are given as above. 
Let $p \in N$ and $0<R<R_0$.
If $(\Si,\partial \Si) \looparrowright (B(p,R),N\cap B(p,R))$ is an immersed (embedded when $n=6$) stable free boundary minimal hypersurface satisfying the area bound: $\Area (\Si \cap B(p, R) )\leq C_0$,
then 
\begin{equation*}
\sup_{x \in \Si \cap B(p, \frac{R}{2})} |A^{\Si}| (x) \leq C_1,
\end{equation*}
where $C_1>0$ is a constant depending on $C_0$, $M$ and $N$.
\end{theorem}

\begin{proof}
The proof is by a contradiction argument which will be divided into three steps. First, if the assertion is false, then we can carry out a blowup argument to obtain a limit after a suitable rescaling. Second, we show that if the limit satisfies certain area growth condition, it has to be a flat hyperplane which would give a contradiction to the choice of the blowup sequence. Finally, we check that the limit indeed satisfies the area growth condition using the monotonicity formula (Theorem \ref{T:monotonicity}). %Recall that $B(p,r) \subset M$ denotes the open geodesic ball of $M$ centered at $p$ with radius $r>0$.

\vspace{0.5em}
{\bf Step 1:} \textit{The blow-up argument.} 

Suppose the assertion is false, then there exists a sequence $(\Si_i,\partial \Si_i) \subset (B(p,R), N\cap B(p,R))$ of immersed (embedded when $n=6$) stable free boundary minimal hypersurfaces such that
\begin{equation}
\label{E:area-bdd}
\Area (\Si_i \cap B(p, R)) \leq C_0,
\end{equation}
but as $i \to \infty$, we have
\[ \sup_{x \in \Si_i \cap B(p, \frac{R}{2})} |A^{\Si_i}|(x) \rightarrow \infty.\]
Therefore, we can pick a sequence of points $x_i \in \Si_i \cap B(p, \frac{R}{2})$ such that $|A^{\Si_i}|(x_i)\rightarrow\infty$. 
By compactness we can assume that $x_i \to x \in B(p,\frac{2R}{3})$. 
By Schoen-Simon-Yau interior curvature estimates \cite{SSY75} (or Schoen-Simon's curvature estimates \cite{SS81} when $n=6$), we must have $x\in N$, and moreover, the connected component of $\Si_i\cap B(p, R)$ that passes through $x_i$ must have a non-empty free boundary component lying on $N\cap B(p, R)$.
Define a sequence of positive numbers 
\[ r_i:=(|A^{\Si_i}|(x_i))^{-\frac{1}{2}},\]
then we have $r_i\rightarrow 0$ and $r_i \, |A^{\Si_i}|(x_i)\rightarrow \infty$ as $i \to \infty$.
Now, choose $y_i \in \Si_i \cap B(x_i, r_i)$ so that it achieves the maximum of
\begin{equation}
\label{E:y_i-a}
 \sup_{y \in \Si_i \cap B(x_i, r_i)}  |A^{\Si_i}|(y) \dist_M(y,\partial B(x_i,r_i)) .
\end{equation}
Let $r_i^{\pr}:=r_i-\dist_M(y_i, x_i)$. Note that $r'_i \to 0$ as $r'_i \leq r_i \to 0$. (See Figure \ref{blow up argument}) Moreover, the same point $y_i \in \Si_i \cap B(x_i, r_i)$ also achieves the maximum of
\begin{equation}
\label{E:y_i-b}
\sup_{y \in \Si_i \cap B(y_i, r'_i)}  |A^{\Si_i}|(y) \dist_M(y,\partial B(y_i,r'_i)).
\end{equation}

\begin{figure}[h]
\centering
\includegraphics[height=1.6in]{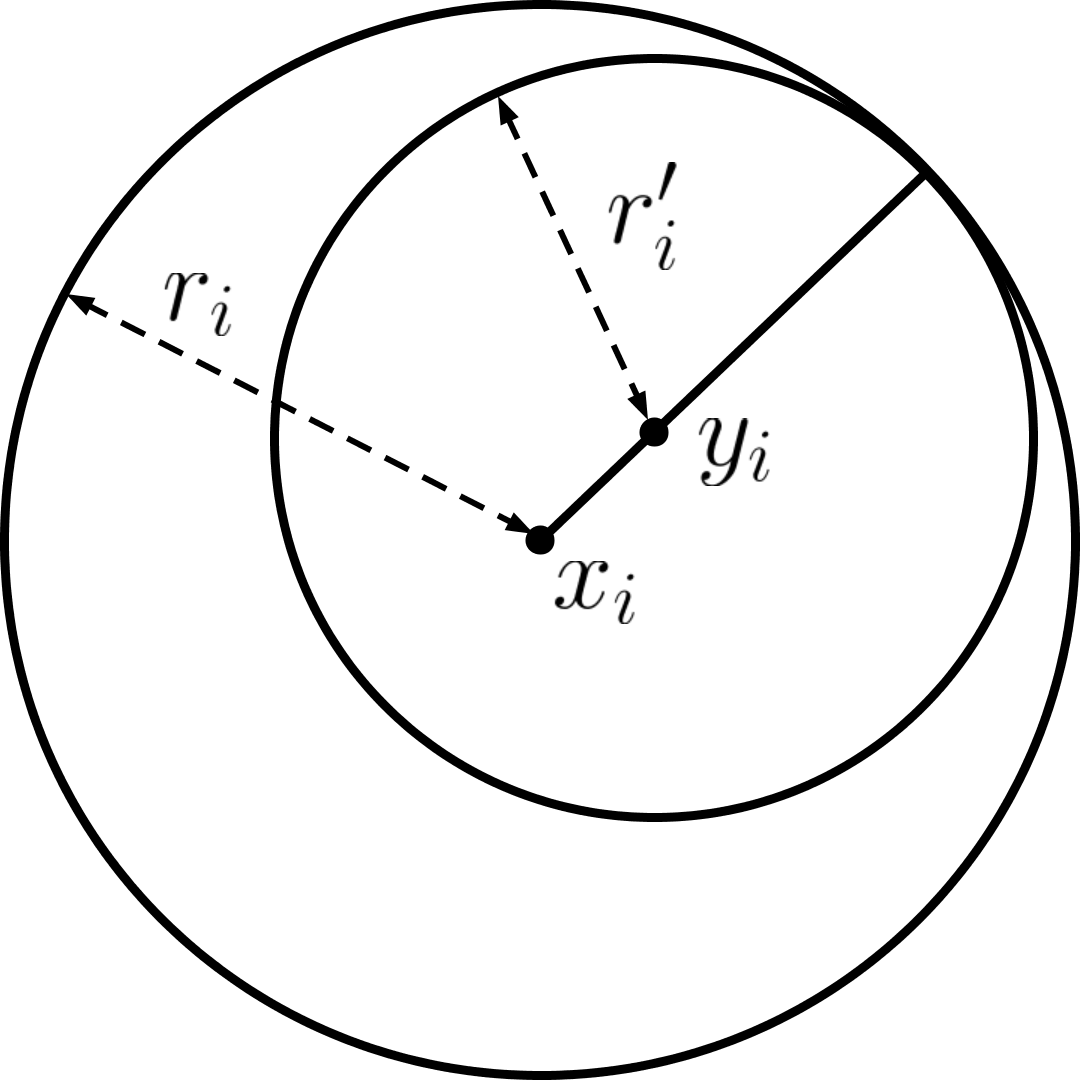}
\caption{$B(x_i, r_i)$ and $B(y_i, r_i^{\pr})$}
\label{blow up argument}
\end{figure}

Define $\la_i:=|A^{\Si_i}|(y_i)$, then we have $\la_i \to \infty$ since $r'_i \to 0$ and
\[
\begin{array}{rcl}
\la_i r_i^{\pr} & = &|A^{\Si_i}|(y_i)\dist_M (y_i, \partial B(y_i, r_i^{\pr}))\\
&=& |A^{\Si_i}|(y_i)\dist_M (y_i, \partial B(x_i, r_i))\\
                        & \geq& |A^{\Si_i}|(x_i)\dist_M(x_i, \partial B(x_i, r_i)) \\
                        &=& r_i |A^{\Si_i}|(x_i) \to +\infty,
\end{array}
\]
where the inequality above follows from (\ref{E:y_i-a}).

Let $\eta_i: \R^L \to \R^L$ be the blow up maps $\eta_i(z):=\la_i (z-y_i)$ centered at $y_i$. Denote $(M'_i,N'_i):=(\eta_i(M),\eta_i(N))$ and $B'(0,r)$ be the open geodesic ball in $M'_i$ of radius $r>0$ centered at $0 \in M'_i$. We get a blow-up sequence of immersed stable free boundary minimal hypersurfaces
\[ (\Si_i^{\pr},\partial \Si'_i) := (\eta_i(\Si_i),\eta_i(\partial \Si_i)) \subset (B'(0,\la_iR),N' \cap B'(0,\la_i R)).\]
Note that we have $|A^{\Si_i^{\pr}}|(0)=\la_i^{-1}|A^{\Si_i}|(y_i)=1$ for every $i$, and 
the connected component of $\Si_i'$ passing through $0$ must have non-empty free boundary lying on $N_i' \cap B'(0,\la_i R)$. 
For each fixed $r>0$, we have $\la_i^{-1} r < r_i'$ for all $i$ sufficiently large since $\la_i r'_i \to +\infty$. Hence, if $x \in \Si'_i \cap B'(0,r)$, then $\eta_i^{-1}(x) \in \Si_i \cap B(y_i,\la_i^{-1}r) \subset  \Si_i \cap B(y_i,r_i')$. Using (\ref{E:y_i-b}), we have
\begin{equation}
\label{E:A-x}
|A^{\Si_i^{\pr}}|(x)\leq\frac{\la_i r_i^{\pr}}{\la_i r_i^{\pr}-r}
\end{equation}
since $\dist_M(\eta_i^{-1}(x),\partial B(y_i,r'_i)) \geq r'_i - \la_i^{-1}r$ for all $i$ sufficiently large (depending on the fixed $r>0$). Note that the right hand side of (\ref{E:A-x}) approaches $1$ as $i \to \infty$.

\vspace{0.5em}
{\bf Step 2:} \textit{The contradiction argument.}

By the smoothness of $M$ and that $y_i \to x \in M$, we clearly have $B'(0, \la_i r_i^{\pr})$ converging to $T_xM$ smoothly and locally uniformly in $\R^L$. However, as $y_i$ does not necessarily lie on $N$, we have to consider two types of convergence scenario:
\begin{itemize}
\item Type I: $\liminf_{i \to \infty} \la_i \dist_{\R^L}(y_i,N) = \infty$,
\item Type II: $\liminf_{i \to \infty} \la_i \dist_{\R^L}(y_i,N) < \infty$.
\end{itemize}
For Type I convergence, the rescaled constraint surface $N' \cap B'(0,\la_i R)$ will escape to infinity as $i \to \infty$ and therefore disappear in the limit. For Type II convergence, after passing to a subsequence, $N' \cap B'(0,\la_i R) \to P$ smoothly and locally uniformly to some $n$-dimensional affine subspace $P \subset \R^L$.

Assume for now that the blow-ups $\Si_i^{\pr}$ satisfy a \emph{uniform Euclidean area growth} with respect to the geodesic balls in $M_i$, i.e., there exists a uniform constant $C_2>0$ such that for each fixed $r>0$, when $i$ is sufficiently large (depending possibly on $r$), we have
\begin{equation}
\label{E:Euclidean-area-growth}
\Area\big(\Si_i^{\pr}\cap B'(0, r)\big)\leq C_2 \, r^n.
\end{equation}
Using either the classical convergence theorem for minimal submanifolds with bounded curvature (for Type I convergence) or Theorem \ref{T:convergence} (for Type II convergence), there exists a 
subsequence of the 
connected component of $\Si_i^{\pr}$ passing through $0$ converging smoothly and locally uniformly to either
\begin{itemize}
\item a complete, immersed stable minimal hypersurface $\Si_{\infty}^1$ in $T_xM$, or
\item a non-compact, immersed stable free boundary minimal hypersurface $(\Si_{\infty}^2,\partial \Si_\infty^2) \subset (T_xM,P)$ such that $\partial\Si_\infty^2 \neq \emptyset$,
\end{itemize}
satisfying the same Euclidean area growth as in (\ref{E:Euclidean-area-growth}) \emph{for all} $r>0$ with $\Si'_i$ replaced by $\Si^1_\infty$ or $\Si^2_\infty$. When $n=6$, $\Si^1_\infty, \Si^2_\infty$ are both embedded by our assumption. In the first case, the classical Bernstein Theorem \cite[Theorem 2]{SSY75} (when $2\leq n\leq 5$) or \cite[Theorem 3]{SS81} (when $n=6$) implies that $\Si^1_\infty$ is a flat hyperplane in $T_xM$, which is a contradiction as $A^{\Si_{\infty}^1}(0)=1$. %(Note that this is the only place where we use the assumption $2\leq n\leq 5$.)  
In the second case, as the constraint hypersurface $P$ is a hyperplane in $T_xM$, we can double $\Si_\infty^2$ as in Lemma \ref{L:reflection} by reflecting across $P$ to obtain a complete, immersed (embedded when $n=6$) stable minimal hypersurface in $T_xM$ with Euclidean area growth. 
%(see Figure \ref{doubling}). 
This gives the same contradiction as in the first case.

\begin{comment}
\begin{figure}[h]
\centering
\includegraphics[height=1.2in]{doubling.png}
\caption{Doubling of $\Si^2_{\infty}$}
\label{doubling}
\end{figure}
\end{comment}

\vspace{0.5em}
{\bf Step 3:} \textit{The area growth condition.}

It remains now to establish the uniform Euclidean area growth for $\Si'_i$ in (\ref{E:Euclidean-area-growth}). This is essentially a consequence of the monotonicity formula (Theorem \ref{T:monotonicity}). In the following, $C_3, C_4, \cdots$ will be used to denote constants depending only on $(M\subset \R^L, N)$.
%The actual arguments are a bit technical as the monotonicity formula only works for balls centered on $N$.
%The actual arguments are a bit technical as we have to switch between intrinsic geodesic balls and Euclidean balls using (\ref{E:ball-comparison}).

Let $d_i:=\dist_M (y_i, N)$ and $z_i \in N$ be the nearest point projection (in $M$) of $y_i$ to $N$. Hence $d_i\to 0$ by the choice of $y_i$. We have to consider two cases:
\begin{itemize}
\item \textit{Case 1:} $\liminf_{i \to \infty} \la_i d_i = \infty$,
\item \textit{Case 2:} $\liminf_{i \to \infty} \la_i d_i < \infty$.
\end{itemize}

Let us first consider \textit{Case 1}. 
Fix $r>0$. Since $\la_i d_i \to \infty$, we have for all $i$ sufficiently large (depending on $r$)
\begin{equation}
\label{E:balls}
B(y_i, \la_i^{-1} r) \subset B(y_i, d_i) \subset B(z_i, 2d_i) \subset B(z_i, \frac{R}{2}) \subset B(p, R).
\end{equation}
Note that $B(y_i,d_i) \cap N=\emptyset$, by the interior monotonicity formula \cite[Theorem 17.6]{Si83} and (\ref{E:balls}), we have for $i$ sufficiently large
\[ \Area (\Si_i\cap B(y_i, \la_i^{-1} r))  \leq C_3 \frac{\Area (\Si_i\cap B(y_i, d_i))}{d_i^n}(\la_i^{-1}r)^n. \]
Using $d_i \to 0$, (\ref{E:balls}) and the boundary monotonicity formula (Theorem \ref{T:monotonicity}), we have for $i$ sufficiently large
\[ \Area (\Si_i\cap B(y_i, \la_i^{-1} r))  \leq 2^{n} C_4 \frac{\Area \big(\Si_i\cap B(z_i, \frac{R}{2})\big)}{(\frac{R}{2})^n}(\la_i^{-1}r)^n. \]
Finally, using (\ref{E:balls}) and (\ref{E:area-bdd}), for $i$ sufficiently large we have
\[ \Area (\Si_i\cap B(y_i, \la_i^{-1} r))  \leq \big(2^{2n}\, C_4 \, C_0 \, R^{-n}\big)\cdot (\la_i^{-1}r)^n,\]
which implies (\ref{E:Euclidean-area-growth}). This finishes the proof for \textit{Case 1}.

Now we consider \textit{Case 2}, i.e. $\la_i d_i$ is uniformly bounded for all $i$. %Since $\la_i \to \infty$, we have $d_i \to 0$. 
By similar argument as above, we have
\[ B(y_i,\la_i^{-1}r) \subset B(z_i, d_i+\la_i^{-1}r) \subset  B(z_i, \frac{R}{2}) \subset B(p,R)\]
for all $i$ sufficiently large (for any fixed $r>0$). By exactly the same arguments as in \textit{Case 1}, we have
\[\Area  (\Si_i\cap B(y_i, \la_i^{-1} r))  \leq C_0 2^n C_5 R^{-n} \left(1+\frac{\la_i d_i}{r}\right)^n \cdot (\la_i^{-1} r)^n. \]
Since $\la_i d_i$ is uniformly bounded, for $r$ sufficiently large independent of $i$, (\ref{E:Euclidean-area-growth}) is satisfied. This proves \textit{Case 2} and thus completes the proof of Theorem \ref{T:main-curvature-estimates}.

\end{proof}

%%%%%%%%%%%%%%%%%%%%%%%%%%%%%%%%%%
% Section 5    	 	  2d-Curvature estimates                                 %
%%%%%%%%%%%%%%%%%%%%%%%%%%%%%%%%%%

\section{Proof of Theorem \ref{T:curvature estimates 2-D}}
\label{S:2d-curvature-estimates}

In this section, we prove Theorem \ref{T:curvature estimates 2-D} using the same blow-up arguments as in the proof of Theorem \ref{T:main-curvature-estimates}. However, since we do not assume a uniform area bound of the minimal surfaces, we may not get a single stable minimal surface in the blow-up limit. Nonetheless, with the extra \emph{embeddedness} assumption, the blow-up sequence would still subsequentially converge to a \emph{minimal lamination}. 
%\sout{Therefore, in the blow-up scheme, instead of getting the limit as a minimal surface, we will have the convergence to a minimal lamination.} 
%A minimal lamination in a 3-manifold is roughly a disjoint collection of embedded minimal surfaces \cite{CM00}. Here we will adapt the notion and define minimal lamination with free boundary.  
Roughly speaking, a minimal lamination in a 3-manifold $M^3$ is a disjoint collection $\mc{L}$ of embedded minimal surfaces $\Lambda$ (called the \emph{leaves} of the lamination) such that $\cup_{\Lambda\in \mc{L}} \Lambda$ is a closed subset of $M$. In \cite{CM13}, Colding and Minicozzi proved that a sequence of minimal laminations with uniformly bounded curvature subsequentially converges to a limit minimal lamination. For our purpose, we will generalize the notion of minimal laminations to include the case with free boundary. 

Throughout this section, we will denote $M^3$ to be a compact $3$-manifold with boundary $\partial M$, and without loss of generality, suppose that $M$ is a compact subdomain of another closed Riemannian $3$-manifold $\wti{M}$. Moreover, we denote the half-space
\[\mathbb{R}_{+}^3:=\{(x^1,x^2,x^3)\in \mathbb{R}^3: x^1\geq 0\},\] 
whose boundary is given by the plane $\mathbb{R}^2_1=\partial \mathbb{R}^3_+=\{x^1= 0\}$. First, let us recall the definition of minimal lamination from \cite{CM13}. 

\begin{definition}[Appendix B in \cite{CM13}]
\label{D:min-lamination}
Let $\Om\subset\wti{M}$ be an open subset. A \emph{minimal lamination of $\Om$} is a collection  $\mc{L}$ of disjoint, embedded, connected minimal surfaces, denoted by $\Lambda$ (called the \emph{leaves} of the lamination) such that $\cup_{\Lambda\in \mc{L}} \Lambda$ is a closed subset of $\Om$. Moreover
\begin{itemize}
\item for each 
 $x \in \Om$, there exists a neighborhood $U$ of $x$ in $\Om$ and a local chart $(U,\Phi)$ with $\Phi(U)\subset \mathbb{R}^3$ so that in these coordinates the leaves in $\mc{L}$ pass through $\Phi(U)$ in slices of the form $(\mathbb{R}^2\times \{t\})\cap \Phi(U)$.
 \end{itemize}
\end{definition}

Now we can define minimal laminations with free boundary. %Note that our definition reduces to Definition \ref{D:min-lamination} when $\partial M=\emptyset$.

\begin{definition}
\label{D:min-lamination-free}
A \emph{minimal lamination of $M^3$ with free boundary on $\partial M$} is a collection  $\mc{L}$ of disjoint, embedded, connected minimal surfaces with (possibly empty) free boundary on $\partial M$, denoted by $\Lambda$, such that $\cup_{\Lambda\in \mc{L}} \Lambda$ is a closed subset of $M$. Moreover, for each $x \in M$, one of the following holds:
\begin{itemize}
\item[(i)] %for each $x \in M \setminus \partial M$, there exists a neighborhood $U$ of $x$ and a local chart $(U,\Phi)$ with $\Phi(U)\subset \mathbb{R}^3$ so that in these coordinates the leaves in $\mc{L}$ pass through $\Phi(U)$ in slices of the form $(\mathbb{R}^2\times \{t\})\cap \Phi(U)$;
$x \in M \setminus \partial M$ and
there exists an open neighborhood $U$ of $x$ in $M\setminus\partial M$ such that $\{\Lambda \cap U : \Lambda \in \mc{L}\}$ is a minimal lamination of $U$;
\item[(ii)] %for each $x\in\partial M$, there exists a relatively open neighborhood $\wti{U}$ of $x$ and a local coordinate chart $(\wti{U}, \wti{\Phi})$ \textcolor{red}{such that either} $\wti{\Phi}(\wti{U})\subset \mathbb{R}^3_+$ and $\wti{\Phi}(\partial M \cap\wti{U})\subset \partial\mathbb{R}^3$ so that in these coordinates the leaves in $\mc{L}$ pass through the chart in slices of the form $(\mathbb{R}^2\times \{t\})\cap \wti{\Phi}(\wti{U})$, \textcolor{red}{or $\wti{\Phi}(\wti{U})\subset \R^3$ so that in these coordinates $\mc{L}$ pass through the chart in slices of the form $(\{t\}\times \mathbb{R}^2)\cap \wti{\Phi}(\wti{U})$}.
$x\in\partial M$ and there exists a relatively open neighborhood $\wti{U}$ of $x$ in $M$ and a local coordinate chart $(\wti{U}, \wti{\Phi})$ such that $\wti{\Phi}(\wti{U})\subset \mathbb{R}^3_+$ and $\wti{\Phi}(\partial M \cap\wti{U})\subset \partial\mathbb{R}^3_+$ so that in these coordinates the leaves in $\mc{L}$ pass through the chart in slices of the form $(\mathbb{R}^2\times \{t\})\cap \wti{\Phi}(\wti{U})$;
\item[(iii)] $x\in\partial M$ and there exists an open neighborhood $U$ of $x$ in $\wti{M}$, such that $\{\Lambda \cap U : \Lambda \in \mc{L}\}$ is a minimal lamination of $U$.
\end{itemize}
\end{definition}
%\begin{remark}
%When $\partial M=\emptyset$, our definition reduces to that in \cite[Appendix B]{CM13}.
%\end{remark}

\begin{remark}
Note that the leaves $\La$ of $\mc{L}$ in Definition \ref{D:min-lamination-free} may not be properly embedded in $M$. For example, $\La$ may touch $\partial M$ in the interior of $\La$ in case (iii).
\end{remark}

In the special case that $M^3=\R_+^3$, by the maximum principle \cite[Corollary 1.28]{CM11} we know that all leaves of $\mc{L}$ are properly embedded (except when $\La=\partial \R_+^3$). Therefore Lemma \ref{L:reflection} implies the following reflection principle for minimal lamination with free boundary.
\begin{lemma}[Lamination reflection principle]
\label{L:lamination reflection}
If $\mc{L}$ is a minimal lamination of  $\R_+^3$ with free boundary on $\partial \R_+^3$, then $\{\Lambda \cup \theta(\Lambda): \Lambda \in \mc{L}\}$ is a minimal lamination of $\R^3$ (in the sense of Definition \ref{D:min-lamination}).
\end{lemma}

We will need the following convergence result. The proof will be postponed until section \ref{S:lamination convergence}. %by adapting the proof of \cite[Proposition B.1]{CM13} to the free boundary case.

\begin{theorem}
\label{T:2d-lamination convergence}
Let $(M^3,g)$ be a compact Riemannian $3$-manifold with boundary $\partial M \neq \emptyset$. If $\mc{L}_i$ is a sequence of minimal laminations of $M$ with free boundary on $\partial M$ of uniformly bounded curvature, i.e. there exists a constant $C>0$ such that
\[ \sup\{|A^{\La}|^2(x): x\in\La\in\mc{L}_i\}\leq C, \] 
then a subsequence of $\mc{L}_i$ converges in the $C^\alpha$ topology for any $\alpha<1$ to a Lipschitz lamination $\mc{L}$ with minimal leaves in $M$ and free boundary on $\partial M$.
\end{theorem}

\vspace{0.5em}

%Now we are ready to give the proof of Theorem \ref{T:curvature estimates 2-D}.

\begin{proof}[Proof of Theorem \ref{T:curvature estimates 2-D}]
%The proof is again a contradiction argument and we shall adopt the same notions as in the proof of Theorem \ref{T:main-curvature-estimates} in section \ref{S:curvature-estimates}. Following the blow-up procedure as in Step 1 of Theorem \ref{T:main-curvature-estimates}, we obtain a sequence $(\Si'_i,\partial \Si'_i)\subset (M'_i,\partial M'_i)$ of properly embedded minimal surfaces with free boundary, where $M'_i=\eta_i(M)$ is the rescale of $M$ centered at $y_i\in M$ by some $\la_i\to\infty$. Moreover,  $\Si'_i$ have curvature bounded from above in any ball of fixed radius for $i$ large enough, and $|A^{\Si'_i}|(0)=1$. 

%By compactness of $M$, $y_i$ converges to some point $x$, which we can assume $x \in \partial M$ again by Schoen's interior curvature estimates \cite{Schoen83} (see also \cite{CM02}). As in the proof of Theorem \ref{T:main-curvature-estimates}, we have to consider two types of convergence scenario:
%\begin{itemize}
%\item Type I: $\liminf_{i \to \infty} \la_i \dist_{\R^L}(y_i, \partial M) = \infty$,
%\item Type II: $\liminf_{i \to \infty} \la_i \dist_{\R^L}(y_i, \partial M) < \infty$.
%\end{itemize}
%For Type I convergence, $M'_i$ converge to the whole tangent space $T_xM$ smoothly and locally uniformly in $\R^L$. For Type II convergence, $M'_i$ converges smoothly and locally uniformly to some $3$-dimensional halfspace $H\simeq\R^{3}_+ \subset T_xM \subset \R^L$.

We follow the same contradiction argument as in the proof of Theorem \ref{T:main-curvature-estimates} and adopt the same notions therein. After a blow-up process, we again face two types of convergence scenario. %In Type I convergence, $M'_i$ converge to the whole tangent space $T_xM$ smoothly and locally uniformly in $\R^L$. For Type II convergence, $M'_i$ converges smoothly and locally uniformly to some $3$-dimensional halfspace $H\simeq\R^{3}_+ \subset T_xM \subset \R^L$.
By Colding-Minicozzi's convergence theorem for minimal laminations with bounded curvature \cite[Proposition B.1]{CM13} (for Type I convergence) and Theorem \ref{T:2d-lamination convergence} (for Type II convergence), a subsequence of blowups %still denoted by $\Si'_i$, 
converges to 
\begin{itemize}
\item a minimal lamination $\ti{\mc{L}}$ in $T_xM\simeq \R^3$, or
\item a minimal lamination $\mc{L}$ in $H$ with free boundary on $\partial H$.
\end{itemize}
%Note that both convergence results hold true for a sequence of manifolds with smoothly converging metrics (in our case $(M'_i, \partial M'_i)$). 
In the second case, we can apply the lamination reflection principle (Lemma \ref{L:lamination reflection}) to obtain a minimal lamination $\ti{\mc{L}}$ in $T_x M \simeq \R^3$. %Since $|A^{\Si'_i}|(0)=1$, 
By the blowup assumption, we know that the origin $0\in\R^3$ is in the support of $\ti{\mc{L}}$, and the curvature of the leaf $\La_0$ passing through $0$ is exactly $1$ at $0$, i.e. $|A^{\La_0}|(0)=1$.

Now we analyze the structure of the minimal lamination $\ti{\mc{L}}\subset\R^3$ for both cases. 
%By a standard argument (c.f. \cite{Li-Zhou16}), if there is an accumulating leaf $\La$ in $\ti{\mc{L}}$, $\La$ (when it is orientable), or its double cover $\ti{\La}$ (when $\La$ is non-orientable) is a complete, stable, minimal surface in $\R^3$. 
We refer to \cite{Li-Zhou16} for well-known terminologies for minimal laminations. If $\La\in \ti{\mc{L}}$ is an accumulating leaf, then either $\La$ or its double cover $\ti{\La}$ is a complete, stable minimal surface in $\R^3$, which must be an affine plane by the Bernstein theorem in $\R^3$ (see \cite{DoPe, FC-Schoen80}). Therefore, the leaf $\La_0$ passing through $0$ must be an isolated leaf. Since all the surfaces in the sequence $\Si'_i$ are stable with free boundary, the smooth convergence of $\Si'_i$ to $\ti{\mc{L}}$ or $\mc{L}$ and the reflection principle (Lemma \ref{L:reflection}) imply that $\La_0$ is a complete, stable, minimal surface in $\R^3$. This again violates the Bernstein theorem as $|A^{\La_0}|(0)=1$ by our construction. Therefore, we arrive at a contradiction and finish the proof of Theorem \ref{T:curvature estimates 2-D}.
\end{proof}

%%%%%%%%%%%%%%%%%%%%%%%%%%%%%%%%%%%%%%%%%%
% Section 6    	 	  Convergence of free boundary minimal submanifolds           %           		     
%%%%%%%%%%%%%%%%%%%%%%%%%%%%%%%%%%%%%%%%%%
\section{Convergence of free boundary minimal submanifolds}
\label{S:convergence}

In this section, we prove a general convergence result (Theorem \ref{T:convergence}) for free boundary minimal submanifolds with uniformly bounded second fundamental form. Note that this convergence result does not require stability and holds in any dimension and codimension. 
%In this appendix, we collect several useful facts about Fermi coordinates and then discuss some convergence results for minimal submanifolds with free boundary and uniformly bounded second fundamental form.

To facilitate our discussion, let us first review some basic properties of \emph{Fermi coordinates}. Let $N^n \subset M^{n+1}$ be an embedded hypersurface (without boundary) in the Riemannian manifold $(M,g)$. We can assume that both $N$ and $M$ are complete. Fix a point $p \in N$, if we let $(x_1,\cdots,x_n)$ be the geodesic normal coordinates of $N$ centered at $p$, and $t=\dist_M(\cdot,N)$ be the signed distance function from $N$ which is well-defined and smooth in a neighborhood of $p$ inside $M$. Therefore, for $r_0>0$ sufficiently small, there exists a diffeomorphism, called a \emph{Fermi coordinate chart},
\[ \phi: B^{n+1}_{r_0}(0) \subset T_p M \to U \subset M \]
\[ (t,x_1,\cdots,x_n) \mapsto \phi(t,x_1,\cdots,x_n), \]
such that $U \cap N = \phi(\{t=0\})$. Here, $B^{n+1}_{r_0}(0)$ is the open Euclidean ball of $T_p M \cong \R^{n+1}$ of radius $r_0>0$ centered at $0$. We refer the readers to \cite[Section 2.2]{LiZ16} for a more detailed discussion on Fermi coordinates. The components of the metric $g$ in Fermi coordinates satisfy $g_{tt}=1$ and $g_{x_i t}=0$ for $i=1,\cdots,n$.

%Fix a point $p$ on the hypersurface $N^n \subset M^{n+1}$, we recall the definition of the (locally defined) Fermi coordinates system $\{x^0, x^1,\cdots, x^n\}$ in a (relatively open) neighborhood $U$ of $p$ in $(M^{n+1}, g)$ such that $x(p)=0$. Using the normal exponential map $\exp^\perp: T^\perp N \rightarrow M$ from the normal bundle $T^\perp N$ to $M$, we can identify a neighborhood $U$ of $p$ with a neighborhood $\wti{U}$ of $p$ in $T^\perp(N)$. By taking $x^0$ as the signed distance function to $N$ and $\{x^1, \cdots, x^n\}$ the geodesic normal coordinates on $N$ centered at $p$, $\{x^0, x^1, \cdots, x^n\}$ forms a coordinate system in $\wti{U}$ when $U$ is small enough.
%we can take the local Fermi coordinates $\{x^0, x^1,\cdots, x^n\}$ of a neighborhood $U$ of $p$ in $(M^{n+1}, g)$, such that under the coordinates, 
%Under the local Fermi coordinates, $N\cap U$ is defined by $x^0=0$, and if the metric coefficients are $\{g_{ab}\}$, $0\leq a, b \leq n$, then $g_{0i}=0$, for $i=1, \cdots, n$.  We can always think of $U$ as a subset of $\R^{n+1}$, and $N\cap U$ is identified with $\{x^0=0\}\cap U$, and the vector $\frac{\partial}{\partial x^0}$ is perpendicular to $N \cap U$ under the metric $\{g_{ab}\}$.

Let $(\Sigma,\partial \Sigma) \subset (M, N)$ be smooth embedded free boundary minimal $k$-dimensional submanifold, with $1 \leq k \leq n$. Fix any $p \in \partial \Sigma \subset N$, and let $\phi:B^{n+1}_{r_0}(0) \to U$ be a Fermi coordinate chart as above centered at $p$. After a rotation we can assume that 
\[ T_p (\partial \Sigma)=\{x_k=\cdots=x_n=0=t\} \cong \R^{k-1}.\]
Since $\Sigma$ meets $N$ orthogonally along $\partial \Si$, after picking a choice on the sign of $t$, the tangent half-space $T_p \Si$ is given by
\[ T_p \Sigma=\{x_k=\cdots=x_n=0, \, t\geq 0\} \cong \R^k_+.\]
Hence, under the Fermi coordinates in a neighborhood of $p$, $\Sigma$ can be written as a graph of $u=(u_1,\cdots,u_{n+1-k})$ which is a $\R^{n+1-k}$-valued function of $(t,x')=(t,x_1,\cdots,x_{k-1})$ in a domain of $\R^k_+$, i.e.
\[ \phi^{-1}(\Sigma) = \{(t,x',u(t,x'))\} \subset \R^{n+1}_+.\]
Moreover, $\phi^{-1}(\partial \Si)$ is given by the same graph with $t=0$. Since $\frac{\partial}{\partial t}$ is a unit normal vector field along $N \cap U$, it is clear that the free boundary condition along $\partial \Si$ is equivalent to 
\begin{equation}
\label{E:freebdy-equation}
\frac{\partial u_\ell}{\partial t} (0,x') =0 \quad \text{ for $\ell=1,\cdots,n+1-k$.}
\end{equation}

We now state the convergence result for free boundary minimal submanifolds with uniformly bounded area and the second fundamental form.

\begin{theorem}
\label{T:convergence}
Suppose we have a sequence $(\Si_j,\partial \Si_j) \subset (M,N)$ of immersed free boundary minimal $k$-dimensional submanifolds, where $1 \leq k \leq n$, with uniformly bounded area and second fundamental form, i.e. there exist positive constants $C_0,C_1>0$ such that
\[ \Area(\Si_j) \leq C_0 \quad \text{ and } \quad \sup_{\Si_j} |A^{\Si_j}| \leq C_1 \]
for all $j$, then after passing to a subsequence, $(\Si_j, \partial \Si_j)$ converges smoothly and locally uniformly to $(\Si_\infty,\partial \Si_\infty) \subset (M,N)$ which is a smooth immersed free boundary minimal $k$-dimensional submanifold.
 %Given a sequence of $(k+1)$-dimensional immersed minimal submanifolds $\{\Si_{j}\}_{j\geq 1}$ in $(U, g)$ with free boundary lying on $N$. Assume that $\Si_j${'}s have uniformly bounded second fundamental forms and area, i.e. $\sup_{\Si_{j}}|A_j|\leq C_1$ and $\Area(\Si_j)\leq C_2$ for all $j$,
%then a subsequence $\{\Si_{j^{\pr}}\}$ will converge smoothly to a $(k+1)$-dimensional immersed minimal submanifold $\Si_{\infty}$ in $U$ with free boundary lying on $N$.
\end{theorem}
\begin{proof}
The convergence away from $N$ follows from the classical convergence results. By the second fundamental form bound, we can cover $N$ by balls (of a uniform size) under Fermi coordinates centered at $p \in N$ such that each $\Si_j$ can be written as graphs over some domain of $T_p \Si_j$ with uniformly bounded gradient (see \cite[\S 2 Lemma 2.4]{CM11}). Using the uniform area bound together with the monotonicity formula (Theorem \ref{T:monotonicity}), there is a uniform upper bound on the number of sheets of the graphs. After passing to a subsequence, the number of sheets remains constant for all $j$ and each sheet is a graph over a $k$-dimensional subspace of $T_p M$ or a $k$-dimensional half-space orthogonal to $T_pN$. The first case again follows from the classical interior convergence result. The second case follows from standard elliptic PDE theory with Neumann boundary conditions (\ref{E:freebdy-equation}) (see \cite{ADN} for example).
%We can cover $U$ by small balls under the Fermi coordinates, such that $\Si_j$'s decompose to graphs over tangent plane with bounded gradients (see \cite[\S 2 Lemma 2.4]{CM11}). For balls center at the points on $N$, the graphs are over tangent planes of $\Si$, which are transverse and orthogonal to $N$. The monotonicity formula (\ref{monotonicity formula}) implies that there are only finitely many such graphs. We can then use Proposition \ref{convergence1} to show the convergence over these small balls. For balls which do not intersect with $N$, we can apply the classical convergence result. Finally a diagonalization argument gives the required subsequence.
\end{proof}

%%%%%%%%%%%%%%%%%%%%%%%%%%%%%%%%%%%%%%%%%%
% Section 7    	 	  Free boundary lamination convergence                                 %           		     
%%%%%%%%%%%%%%%%%%%%%%%%%%%%%%%%%%%%%%%%%%

\section{Convergence of free boundary minimal lamination}
\label{S:lamination convergence}

Finally, we give the proof of Theorem \ref{T:2d-lamination convergence} which was used in section \ref{S:2d-curvature-estimates}.

\begin{proof}[Proof of Theorem \ref{T:2d-lamination convergence}]
For simplicity we will assume that each lamination $\mc{L}_i$ has finitely many leaves where the number of leaves may depend on $i$; this will suffice for our application.
For any interior point $x \in M \setminus \partial M$, the argument used in the proof of \cite[Proposition B.1]{CM13} implies the convergence in a small neighborhood of $x$ in $M \setminus \partial M$. Hence, we only need to deal with the convergence near a boundary point $x \in \partial M$.

Fix $p\in \partial M$ and let $N=\partial M$. The theorem will follow once we construct uniform coordinate charts in a small neighborhood of $p$ in the Fermi coordinate system as in section \ref{S:convergence}.
Let $\varphi$ be a Fermi coordinate chart in a relatively open neighborhood $U$ of $p$ in $M$, i.e.,
\[\varphi : U \subset M \to \wti{U}\subseteq \mathbb{R}_{+}^3,\]
such that $\varphi (p)=0$ and $\varphi(N\cap U)=\{x_1=0\}\cap \wti{U}$.
Here, $(x_1,x_2,x_3)$ are the local Fermi coordinate system centered at $p$ (i.e. $t=x_1$). Suppose that $B_{4r_0}^+ \subset \wti{U}$ for some small $r_0$ to be chosen later,  where $B_{4r_0}^+=B_{4r_0}\cap \{x_1\geq 0\}$ denotes the half ball in $\mathbb{R}_+^3$ with radius $4r_0$ centered at the origin. 

Next, we will construct uniform coordinate charts on $\varphi^{-1}(B_{r_0}^+)$. %Since each lamination has uniformly bounded curvature, there exists $C>0$ so that 
Note that for each $i$ and every $\Lambda\in \mc{L}_i$, we have
$\sup_{\Lambda}|A^\La|^2\leq C.$
We may choose $r_0$ sufficiently small so that $Cr_0$ is as small as we wish. 
Then for each fixed $i$, 
\[\bigcup_{\Lambda\in \mc{L}_i} \varphi(\Lambda\cap U)\cap B_{4r_0}^+  \] 
gives a finite number of disconnected surfaces with bounded curvature in the Fermi coordinate system. 

Since the lamination has uniformly bounded curvature, by the tilt estimates as in the proof of \cite[Lemma 2.11]{CM11}, there exists a constant $\delta>0$ such that for each lamination $\mathcal{L}_i$, 
we have the following two cases: (i) none of the leaves of $\mathcal{L}_i$ meets $\partial \mathbb{R}_{+}^3$ in %the half ball 
$B_{\delta r_0}^+$ (except possibly for one leaf touching $\partial \R_+^3$ tangentially at some points); (ii) there exists a leaf of $\mathcal{L}_i$ meeting $\partial \mathbb{R}_{+}^3$ along some non-empty free boundary. For case (i), we can construct uniform coordinate charts as in the proof of \cite[Proposition B.1]{CM13} in a neighborhood of the larger manifold $\wti{M}$. For case (ii), we claim that in %the ball 
$B_{2\delta r_0}^+$, all leaves of $\mathcal{L}_i$ which intersect $B_{\delta r_0}^+$ must meet $\partial \R_+^3$ along some non-empty free boundary; otherwise, the tilt estimates will imply that two leaves intersect somewhere in $B_{r_0}^+$ which contradicts the assumption that all leaves are disjoint. Note that the tilt estimates in \cite[Lemma 2.11]{CM11} only use the uniform curvature bound of leaves in $\mc{L}_i$, but not the minimal surface equations.

Now, we  focus on case (ii). For simplicity, we  use $r_0$ to denote $\delta r_0$. The free boundary condition %satisfied by $\La$ 
and the choice of Fermi coordinates imply that these surfaces meet $\partial \R^3_+$ orthogonally in the Euclidean metric. Going to a further subsequence (possibly with $r_0$ even smaller), for  fixed $i$, every sheet of
\[\bigcup_{\Lambda\in \mc{L}_i} \varphi(\Lambda\cap U)\cap B_{2r_0}^+, \] 
which intersects $B_{r_0}^+$ is a graph with small gradient over a subset of certain fixed plane perpendicular to $\partial \R^3_+$ (which can be chosen as $\R^2\times \{0\}:=\{x^3=0\}$ after a rotation keeping $\partial \R^3_+$ fixed as a set) containing a half ball of radius $r_0$ (see \cite[Lemma 2.4]{CM11}).

%\textcolor{blue}{Now we will consider two cases: the first case is that these surfaces do not meet $\partial \mathbb{R}_{+}^3$; and the second case is that these surfaces have boundary on $\partial \mathbb{R}_{+}^3$. Note that for each lamination, only one case can happen (possibly with $r_0$ even smaller) since each lamination has uniformly bounded curvature. For the first case, we can construct uniform coordinate charts as in the proof of \cite[Proposition B.1]{CM13}. Hence, we will mainly focus on the second case.  For the second case, the free boundary condition satisfied by $\La$ and the choice of Fermi coordinates imply that these surfaces meet $\partial \R^3_+$ orthogonally under the Euclidean metric. Going to a further subsequence (possibly with $r_0$ even smaller), for each fixed $i$, every sheet of
%\[\bigcup_{\Lambda\in \mc{L}_i} \varphi(\Lambda\cap U)\cap B_{2r_0}^+ \] 
%is a graph with small gradient over a subset of certain fixed plane perpendicular to $\partial \R^3_+$ (which can be chosen as $\R^2\times \{0\}:=\{x^3=0\}$ after a rotation keep $\partial \R^3_+$ fixed) containing a half ball of radius $r_0$ (see \cite[Lemma 2.4]{CM11}). }

We will show that in a concentric half ball of smaller radius in %this half ball 
$B_{2r_0}^+$, the sequence of laminations converges in the $C^\alpha$ topology to a lamination for any $\alpha<1$. The coordinate chart $\Phi$ required by the definition of a lamination will be given by the Arzela-Ascoli theorem as a limit of a sequence of bi-Lipschitz maps
\[ \Phi_i:B_{2r_0}^+\to \mathbb{R}^3_+\] 
with bounded bi-Lipschitz constants, and $\Phi$ will be defined on a slightly smaller concentric half ball $B_{sr_0}^+$ for some $s>0$ to be determined. Furthermore, we will show that for each $i$ fixed
\[\Phi_i\big(B_{sr_0}^+\cap \varphi(\cup_{\Lambda\in \mc{L}_i}\Lambda\cap U)\big)\] 
is the union of subsets of planes which are each parallel to $\mathbb{R}^2\times\{0\}\subseteq \mathbb{R}^3_+.$

Set the map $\Phi_i$ by letting 
\[\Phi_i^{-1}(y_1,y_2,y_3)=(y_1,y_2,\phi_i(y_1,y_2,y_3)),\] where $\phi_i$ is defined as follows: order the sheets of $B_{2r_0}^+\cap \varphi(\cup_{\Lambda\in \mc{L}_i} \Lambda\cap U)$ as $\Lambda_{i,k}$  for $k=1,\ldots$ by increasing values of $x_3$ and let $\Lambda_{i,k}$  be the graph of the function $f_{i,k}$ over (part of) the $\mathbb{R}^2\times \{0\}$ plane. In the following we only need to consider those sheets $\Lambda_{i, k}$ where $\Lambda_{i, k}\cap B^+_{r_0}\neq \emptyset$, since we eventually will work on a much smaller concentric half ball. The domain of such $f_{i,k}$ contains the half ball of radius $r_0$ centered at the origin of the $\mathbb{R}^2\times\{0\}$ plane. Again as $Cr_0$ can be chosen small enough, we can assume that $|\nabla f_{i, k}|$ are as small as we want.
Moreover, the free boundary condition satisfied by $\Lambda_{i, k}$ is equivalent to the Neumann boundary condition:
\begin{equation}
\frac{\partial f_{i,k}(0,\cdot)}{\partial x_1} =0.
\end{equation}

Set $w_{i,k}=f_{i,k+1}-f_{i,k}$. In the following, $\Delta$, $\nabla$, and $\text{div}$ will be with respect to the Euclidean metric on $\mathbb{R}^2\times\{0\}$. By a standard computation (cf. \cite[Chapter 7]{CM11} or \cite[(7)]{Si87}), we have
\begin{equation}\label{equ:2d:main:diff}
\text{div}((a+Id)\nabla w_{i,k})+b\nabla w_{i,k}+cw_{i,k}=0,\,\text{ and }\,\,\frac{\partial w_{i,k}(0,\cdot)}{\partial x_1}=0,
\end{equation}
where $a$ is a matrix-valued function, $b$ is a vector-valued function, and $c$ is simply a real-valued function.

Note that $a$, $b$, and $c$ depend on $i$,  but the norms of $a, b, c$ can be made uniformly small if $Cr_0$ is small enough and if we rescale our ambient manifold by a large factor. By (\ref{equ:2d:main:diff}), and the Harnack inequality (see \cite[8.20]{GiTr} and \cite[Section 6]{ADN}) applied to the positive function $w_{i,k}$ gives
\begin{equation}\label{equ:2d:main:harnack}
%sr_0\sup_{\mathbb{B}_{sr_0}^+} |\nabla w_{i,k}|\leq C
 \sup_{\mathbb{B}_{2sr_0}^+} w_{i,k}\leq %\text{exp}(\eps_0 s^\beta) 
 C_1 \inf_{\mathbb{B}_{2sr_0}^+}w_{i,k},
\end{equation}
where $C_1$ depends only on the norms of $a, b$ and $c$. 
Here, $\mathbb{B}_t^+$ is the half ball in $\mathbb{R}^2\times \{0\}$ with radius $t$ and center 0. Set $\mathbf{M}_{i,k}=f_{i,k}(0,0)$. In the region \[\{(y_1,y_2,y_3)\in \mathbb{B}_{r_0}^+\times [\mathbf{M}_{i,k},\mathbf{M}_{i,k+1}]\},\] define the function $\phi_i$ by
\[\phi_i(y_1,y_2,y_3)=f_{i,k}(y_1,y_2)+\frac{y_3-\mathbf{M}_{i,k}}{\mathbf{M}_{i,k+1}-\mathbf{M}_{i,k}}w_{i,k}(y_1,y_2).\]
Hence, \[\Phi_i^{-1}\big(y_1,y_2,f_{i,k}(0,0)\big)=\big(y_1,y_2,f_{i,k}(y_1,y_2)\big);\] that is, $\Phi_i$ maps $\Lambda_{i,k}$ to a subset of the plane $\mathbb{R}^2\times\{f_{i,k}(0,0)\}$.

Note that $\phi_i(0, 0, 0)=0$.
Moreover, we have
\begin{equation}\label{equ:2d:main:grad_phi}
\nabla \phi_i= \nabla f_{i,k}+ \frac{y_3-\mathbf{M}_{i,k}}{\mathbf{M}_{i,k+1}-\mathbf{M}_{i,k}} \nabla w_{i,k} + \frac{w_{i,k}}{\mathbf{M}_{i,k+1}-\mathbf{M}_{i,k}}\frac{\partial}{\partial y_3}.
\end{equation}
By (\ref{equ:2d:main:harnack}) and (\ref{equ:2d:main:grad_phi}), we know that for each $i$ the map $\Phi_i$ restricted to $B_{sr_0}^+\subseteq \mathbb{R}_+^3$ is bi-Lipschitz with uniformly bounded bi-Lipschitz constant.

By the Arzela-Ascoli theorem, a subsequence of $\Phi_i$ converges in the $C^\alpha$ topology for any $\alpha<1$ to a Lipschitz coordinate chart $\Phi$ with the properties that are required. 
By standard elliptic regularity theory, the leaves are either minimal surfaces (for the first case) or minimal surfaces  with free boundary on $N$ (for the second case).
\end{proof}

%%%%%%%%%%%%%%%%%%%%%%%%%%%%%%%%%%%%%%%%%%%%%%%%%%%%%%%%%%%%%%%%%%%%%%%%%%%%%%%%%%%%%%%%%%

\bibliographystyle{amsplain}
\bibliography{reference-curv-estimate}

\end{document}